\documentclass[regno]{amsart}
\usepackage{latexsym,amsmath,amssymb,amscd}
\usepackage[all]{xy}
\usepackage{tikz-cd}
\usepackage{enumerate}
\usepackage{hyperref} 
\usepackage{mathrsfs}
\usepackage{amsfonts}
\usepackage{setspace}

\makeatletter
\@addtoreset{figure}{section}
\def\thefigure{\thesection.\@arabic\c@figure}
\def\fps@figure{h,t}
\@addtoreset{table}{bsection}

\def\thetable{\thesection.\@arabic\c@table}
\def\fps@table{h, t}
\@addtoreset{equation}{section}

\makeatother

\makeatletter
\newcommand\@dotsep{4.5}
\def\@tocline#1#2#3#4#5#6#7{\relax
	\ifnum #1>\c@tocdepth 
	\else
	\par \addpenalty\@secpenalty\addvspace{#2}%
	\begingroup \hyphenpenalty\@M
	\@ifempty{#4}{%
		\@tempdima\csname r@tocindent\number#1\endcsname\relax
	}{%
		\@tempdima#4\relax
	}%
	\parindent\z@ \leftskip#3\relax \advance\leftskip\@tempdima\relax
	\rightskip\@pnumwidth plus1em \parfillskip-\@pnumwidth
	#5\leavevmode\hskip-\@tempdima #6\relax
	\leaders\hbox{$\m@th
		\mkern \@dotsep mu\hbox{.}\mkern \@dotsep mu$}\hfill
	\hbox to\@pnumwidth{\@tocpagenum{#7}}\par
	\nobreak
	\endgroup
	\fi}
\makeatother 

\newtheorem{theorem}{Theorem}
\newtheorem{corollary}[theorem]{Corollary}
\newtheorem{definition}[theorem]{Definition}
\newtheorem{example}[theorem]{Example}

\newtheorem{lemma}[theorem]{Lemma}

\newtheorem{proposition}[theorem]{Proposition}
\newtheorem{remark}[theorem]{Remark}

\numberwithin{theorem}{section}
\numberwithin{equation}{section}

\title[Injective envelopes for locally $C^{\ast}$-algebras]{Injective
envelopes for locally $C^{\ast}$-algebras}
\author[1]{Maria Joi\c ta$^\ast$}
\address{Department of Mathematics, Faculty of Applied Sciences, University
Politehnica of Bucharest, 313 Spl. Independentei, 060042, Bucharest, Romania
and Department of Mathematics, Faculty of Mathematics and Computer Science,
University of Bucharest, Str. Academiei nr. 14, Bucharest, Romania}
\email{maria.joita@upb.ro and mjoita@fmi.unibuc.ro }
\thanks{$^\ast$Corresponding author:  \texttt{maria.joita@upb.ro}}
\urladdr{http://sites.google.com/a/g.unibuc.ro/maria-joita}

\author[2]{Gheorghe-Ionu\c{t} \c{S}imon}
\address{Department of Mathematics, Faculty of Applied Sciences, University
Politehnica of Bucharest, 313 Spl. Independentei, 060042, Bucharest, Romania 
}
\email{ionutsimon.gh@gmail.com}


\begin{document}

\begin{abstract}
We introduce the notion of admissible injective envelope for a locally $C^{\ast}$-algebra and show that each object in the category whose objects are unital Fr\'{e}chet locally $C^{\ast}$-algebras and whose morphisms are unital admissible local completely positive maps has a unique admissible injective envelope. The concept of admissible injectivity is stronger than that of injectivity. As a consequence, we show that a unital Fr\'{e}chet locally $W^*$-algebra is injective if and only if the $W^*$-algebras from its Arens-Michael decomposition are injective.
\end{abstract}

\maketitle

\section{Introduction}

Injectivity is a categorical concept. An object $I$ in a category $\mathbf{C}$ is injective if for any two objects $E\subseteq F$ from $%
\mathbf{C}$, any morphism $\varphi :E\rightarrow I$ extends to a
morphism $\widetilde{\varphi}:F\rightarrow I$. Cohen \cite{C} considered
the category whose objects are Banach spaces and whose morphisms are
contractive linear maps. He introduced the notion of injective envelope for
a Banach space and showed that each Banach space has a unique injective
envelope. Hamana \cite{H} proved a $C^{\ast}$-algebraic version of these
results. He considered the category whose objects are unital $C^{\ast}$-algebras and whose morphisms are unital completely positive linear maps. A
unital $C^{\ast}$-algebra $\mathcal{A}$ is injective if for any unital $
C^{\ast}$-algebra $\mathcal{C}$ and a self-adjoint subspace $\mathcal{S}$ of $
\mathcal{C}$ containing the unit, any unital completely positive linear
map $\varphi :\mathcal{S}\rightarrow \mathcal{A}$ extends to a unital
completely positive linear map $\widetilde{\varphi }:\mathcal{C}\rightarrow 
\mathcal{A}$. By Arveson's extension theorem, the $C^{\ast}$-algebra $B(\mathcal{H})$ of all bounded linear operator on a Hilbert space $\mathcal{H}$
is injective. An extension of $\mathcal{A}$ is a pair $(\mathcal{B},\Phi )$
of a unital $C^{\ast}$-algebra $\mathcal{B}$ and a $\ast $-monomorphism $
\Phi :\mathcal{A}\rightarrow \mathcal{B}$. It is injective if $\mathcal{B}$
is injective. According to the Gelfand-Naimark theorem, for any unital $
C^{\ast}$-algebra $\mathcal{A}$ there exist a Hilbert space $\mathcal{H}$
and an isometric $\ast $-morphism $\Phi :\mathcal{A}\rightarrow B\left( 
\mathcal{H}\right) $, and so, each unital $C^{\ast }$-algebra $\mathcal{A}$
has an injective extension. An injective envelope for $\mathcal{A}$ is an
injective extension $(\mathcal{B},\Phi )$ with the property that $ id_{
\mathcal{B}}$ is the unique unital completely positive linear map which
fixes the elements of $\Phi \left( \mathcal{A}\right) $. He showed that any
unital $C^{\ast}$-algebra $\mathcal{A}$ has a unique injective envelope in
the sense that if $(\mathcal{B}_{1},\Phi _{1})$ and $(\mathcal{B}_{2},\Phi
_{2})$ are two injective envelopes for $\mathcal{A}$, there exists a unique $
\ast $-isomorphism $\Psi :\mathcal{B}_{1}\rightarrow \mathcal{B}_{2}$ such
that $\Psi \circ \Phi _{1}=\Phi _{2}$.

In this paper, we propose to extend the Hamana's results in the context of
locally $C^{\ast}$-algebras. In the literature, the locally $C^{\ast }$-algebras are studied under different names like pro-$C^{\ast}$-algebras
(D. Voiculescu \cite{V}, N.C. Philips \cite{P}), $LMC^{\ast}$-algebras (K.
Schm\"{u}dgen \cite{S}), $b^{\ast}$-algebras (C. Apostol \cite{A}) and
multinormed $C^{\ast}$-algebras (A. Dosiev \cite{D}). The term locally $
C^{\ast }$-algebra is due to A. Inoue \cite{I}. A locally $C^{\ast}$
-algebra is a complete Hausdorff complex topological $\ast $-algebra $
\mathcal{A}$ whose topology is determined by an upward filtered family of $
C^{\ast}$-seminorms $\lbrace p_{\lambda }\rbrace_{\lambda \in \Lambda }$. A Fr\'{e}chet locally $C^{\ast}$-algebra is a locally $C^{\ast}$-algebra whose
topology is determined by a countable family of $C^{\ast}$-seminorms. An
element $a\in \mathcal{A}$ is called \textit{local positive} if $a=b^{\ast
}b+c,$ where $b,c\in \mathcal{A}$ such that $p_{\lambda }\left( c \right)
=0 $ for some $\lambda \in \Lambda $. In this case, we say that $a$ is $\lambda $\textit{-positive}. A linear map $\varphi$ from a locally $C^{\ast}$-algebra $\left(
\mathcal{A},\lbrace p_{\lambda }\rbrace_{\lambda \in \Lambda }\right)$ to another locally $
C^{\ast}$-algebra $\left(\mathcal{B},\lbrace q_{\delta }\rbrace_{\delta \in \Delta }\right)$ is 
\textit{local positive} if for each $\delta \in \Delta ,$ there exists $
\lambda \in \Lambda $ such that $\varphi \left( a\right) $ is $\delta $
-positive whenever $a$ is $\lambda $-positive, and  $\delta $-null if $a$ is $\lambda $-null. It is \textit{local completely positive} if 
for each $\delta \in \Delta ,$ there exists $\lambda \in \Lambda $ such that 
$\left[ \varphi \left( a_{ij}\right) \right] _{i,j=1}^{n}$ is $\delta $
-positive, respectively $\delta $-null, in $M_{n}\left( \mathcal{B}\right) $
, the locally $C^{\ast}$-algebra of all matrices of size $n$ with elements
in $\mathcal{B}$, whenever $\left[ a_{ij}\right] _{i,j=1}^{n}$ is $\lambda $
-positive, respectively $\lambda $-null, in $M_{n}\left( \mathcal{A}\right) $
for all positive integers $n$.

A unital locally $C^{\ast}$-algebra $\mathcal{A}$ is \textit{injective} if
for any unital \textcolor{blue}{locally} $C^{\ast}$-algebra $\mathcal{C}$ and self-adjoint subspace $
\mathcal{S}$ of $\mathcal{C} $ containing the unit, any unital local
completely positive linear map $\varphi :\mathcal{S}\rightarrow \mathcal{A}$
extends to a unital local completely positive linear map $\widetilde{\varphi 
}:\mathcal{C}\rightarrow \mathcal{A}$.

Let $(\Delta ,\leq )$ be a directed poset. A \textit{quantized domain} in a
Hilbert space $\mathcal{H}$ is a triple $\lbrace\mathcal{H};\mathcal{E};\mathcal{D
}_{\mathcal{E}}\rbrace$, where $\mathcal{E}=\lbrace\mathcal{H}_{\delta }: \delta \in
\Delta \rbrace$ is an upward filtered family of closed subspaces with dense union 
$\mathcal{D}_{\mathcal{E}}=\bigcup\limits_{\delta \in \Delta }\mathcal{H}
_{\delta}$ in $\mathcal{H}$. If $\Delta $ is countable, we say that $\lbrace
\mathcal{H};\mathcal{E};\mathcal{D}_{\mathcal{E}}\rbrace$ is a Fr\'{e}chet
quantized domain in $\mathcal{H}$. The collection of all linear operators $T:
\mathcal{D}_{\mathcal{E}}\rightarrow \mathcal{D}_{\mathcal{E}}$ such that $T(
\mathcal{H}_{\delta })\subseteq \mathcal{H}_{\mathcal{\delta }},T(\mathcal{H}
_{\delta }^{\bot }\cap \mathcal{D}_{\mathcal{E}})\subseteq \mathcal{H}
_{\delta }^{\bot }\cap \mathcal{D}_{\mathcal{E}}$ and $ T\restriction _{
\mathcal{H}_{\delta }}\in B(\mathcal{H}_{\delta })$ for all $\delta \in
\Delta ,$ denoted by $C^{\ast}(\mathcal{D}_{\mathcal{E}}),$ is a locally $
C^{\ast}$-algebra with the involution given by $T^{\ast }= T^{\bigstar }\restriction_{\mathcal{D}_{\mathcal{E}}}$, where $T^{\bigstar}$
is the adjoint of the unbounded linear operator $T$, and the topology given
by the family of $C^{\ast}$-seminorms $\lbrace\left\Vert \cdot \right\Vert
_{\delta }\rbrace_{\delta \in \Delta }$, where $\left\Vert T\right\Vert_{\delta
}=\left\Vert  T\restriction_{\mathcal{H}_{\delta }}\right\Vert _{B(
\mathcal{H}_{\delta })}$. For every locally $C^{\ast}$-algebra $\mathcal{A}$
whose topology is defined by the family of $C^{\ast }$-seminorms $
\lbrace p_{\lambda }\rbrace_{\lambda \in \Lambda }$, there exists a quantized
domain $\lbrace\mathcal{H};\mathcal{E}=\lbrace \mathcal{H}_{\lambda }\rbrace_{\lambda \in
\Lambda };\mathcal{D}_{\mathcal{E}}\rbrace$ and a local isometric $\ast $
-morphism $\pi :A\mathcal{\rightarrow }C^{\ast}(\mathcal{D}_{\mathcal{E}})$
, that is a $\ast $-morphism such that $\left\Vert \pi \left( a\right)
\right\Vert _{\lambda}=p_{\lambda }\left( a\right) \ $ for all $a\in A$ and
for all $\lambda \in \Lambda $. Therefore, a locally $C^{\ast}$-algebra can
be identified with a $\ast $-subalgebra of unbounded linear operators on a
Hilbert space. In the local convex theory, the locally $C^{\ast}$-algebra $
C^{\ast }(\mathcal{D}_{\mathcal{E}})$ plays the role of $B(\mathcal{H})$ in a certain sense. Dosiev \cite{D} proved a local convex version of
Arveson's extension theorem in the case of unital Fr\'{e}chet locally $
C^{\ast}$-algebras, and showed that if $\lbrace \mathcal{H};\mathcal{E}=\lbrace\mathcal{
H}\rbrace_{n\geq 1};\mathcal{D}_{\mathcal{E}}\rbrace$ is a Fr\'{e}chet
quantized domain in $\mathcal{H}$, then $C^{\ast}(\mathcal{D}_{\mathcal{E}
}) $ is injective. 

In \cite{DD}, Dosiev considers the category of local operator spaces and local
completely contractive maps. He investigates the connection between the
injectivity in this category and the injectivity in the normed case and
shows that the injectivity of a locally $C^{\ast }$-algebra $\left( \mathcal{A},\left\lbrace p_{\lambda} \right\rbrace _{\lambda \in \Lambda }\right) $
implies the injectivity of the $C^{\ast }$-algebra $b(\mathcal{A})$ of all
its bounded elements (i.e. $a\in \mathcal{A}$ is bounded if $\sup\lbrace p_{\lambda}(a)  
:\lambda \in \Lambda \rbrace <\infty $) \cite[Proposition 3.1]{DD}. In general, the converse implication is not true. He proves that in
the case of Fr\'{e}chet locally $W^{\ast }$-algebras (i.e. a Fr\'{e}chet locally $
W^{\ast }$-algebra is an inverse limit of a countable inverse system of $
W^{\ast }$-algebras whose connecting maps are $W^{\ast }$-morphisms), $
\mathcal{A}$ is injective if and only if $b(\mathcal{A})$ is injective \cite[Theorem 4.1]{DD}.
Also, he introduces the notions of $\mathcal{R}$-injectivity and injective $\mathcal{R}$-envelope for a local operator space. The notion of $\mathcal{R}$-injectivity is stronger than the notion of injectivity. In the case of Fr\'{e}chet locally $W^*$-algebras, these two notions coincide.

In this paper, we consider the category whose objects are unital Fr\'{e}chet locally $C^{\ast}$-algebras and whose morphisms are unital admissible local completely positive maps. An injective object in this category is called admissible injective. A linear map $\varphi$ from a locally $C^*$-algebra $\left( \mathcal{A}, \lbrace p_{\lambda} \rbrace_{\lambda\in\Lambda} \right)$ to another locally $C^*$-algebra $\left( \mathcal{B}, \lbrace q_{\lambda} \rbrace_{\lambda\in\Lambda} \right)$ is \textit{admissible local completely positive} if for each $\lambda\in\lambda$, $[\varphi(a_{ij})]_{i,j=1}^{n}$ is $\lambda$-positive, respectively $\lambda$-null in $M_{n}(\mathcal{B})$ whenever $[a_{ij}]_{i,j=1}^{n}$ is $\lambda$-positive, respectively $\lambda$-null, in $M_{n}(\mathcal{A})$ for all positive integer $n$. The notion of admissible injectivity is stronger than the notion of injectivity,but it is weaker than the notion of $\mathcal{R}$-injectivity. In the case of Fr\'{e}chet locally $W^*$-algebras, these two notions coincide with the notion of injectivity.

Following Hamana \cite{H}, we show that any unital Fr\'{e}chet locally $C^*$-algebra has an admissible injective envelope which is unique up to a local isometric $*$-isomorphism (Theorem \ref{p}). First, we introduce the notions of the family of $\mathcal{B}$-seminorms, respectively admissible $\mathcal{B}$-projections, on a unital locally $C^*$-algebra $\mathcal{A}$ which contains $\mathcal{B}$ as a unital locally $C^*$-subalgebra. We prove the existence of a minimal family of $\mathcal{B}$-seminorms on a unital Fr\'{e}chet locally $C^*$-algebra $\mathcal{A}$ (Theorem \ref{Help}). Then we show that the admissible injective envelope of a unital Fr\'{e}chet locally $C^*$-algebra $\mathcal{A}$ is the range of an admissible $\mathcal{A}$-projection. Finally, in Section 6, we show that the admissible injective envelope of a unital Fr\'{e}chet locally $C^*$-algebra can be identified with the inverse limit of the injective envelopes for its Arens-Michael decomposition, and a unital Fr\'{e}chet locally $W^*$-algebra is injective if and only if it is an inverse limit of injective $W^*$-algebras.

\section{Preliminaries}

\subsection{Locally $C^{\ast }$-algebras}

Let $\mathcal{A}$ be a $\ast $-algebra with unit, denoted by $1_{\mathcal{A}
} $. A seminorm $p$ on $\mathcal{A}$ is called \textit{sub-multiplicative}
if $p(1_{\mathcal{A}})=1$ and $p(ab)\leq p(a)p(b)$ for all $a,b\in \mathcal{A
}$. A sub-multiplicative seminorm $p$ on $\mathcal{A}$ is called a \textit{$
C^{\ast}$-seminorm} if $p(a^{\ast }a)=p(a)^{2}$ for all $a\in \mathcal{A}.$

Let $\left( \Lambda ,\leq \right) $ be a directed poset and let $\lbrace
p_{\lambda }\rbrace_{\lambda \in \Lambda }$ be a family of $C^{\ast}$
-seminorms defined on some $\ast $-algebra $\mathcal{A}$. We say that $
\lbrace p_{\lambda }\rbrace_{\lambda \in \Lambda }$ is an \textit{upward
filtered family} of $C^{\ast}$-seminorms if $p_{\lambda _{1}}(a)\leq
p_{\lambda _{2}}(a)\ $ for all $\ a\in \mathcal{A}\ $ whenever $\ \lambda
_{1}\leq \lambda _{2}\ $in$\ \Lambda $. A \textit{locally }$C^{\ast }$
\textit{-algebra} is a complete Hausdorff complex topological $\ast $
-algebra $\mathcal{A}$ whose topology is determined by an upward filtered
family of $C^{\ast}$-seminorms $\lbrace p_{\lambda }\rbrace_{\lambda \in \Lambda}$. A
metrizable locally $C^{\ast}$-algebra is called a \textit{Fr\'{e}chet} 
\textit{locally $C^{\ast}$-algebra.}
Furthermore, note that any $C^*$-algebra may be regarded as locally $C^*$-algebra.

An element $a\in \mathcal{A}$ is \textit{bounded} if $\sup \lbrace p_{\lambda
}\left( a\right) :\lambda \in \Lambda \rbrace <\infty $. Then $b\left( \mathcal{A}
\right) :=\left\lbrace a\in \mathcal{A}: \left\Vert a\right\Vert _{\infty }:=\sup
\lbrace p_{\lambda }\left( a\right) :\lambda \in \Lambda \rbrace <\infty \right\rbrace $ is a $
C^{\ast}$-algebra with respect to the $C^{\ast}$-norm $\left\Vert \cdot
\right\Vert _{\infty}$. Moreover, $b\left( \mathcal{A}\right) $ is dense in 
$\mathcal{A}$ \cite[Proposition 1.11]{P}.

We see that a locally $C^{\ast}$-algebra $\mathcal{A}$ can be
realized as a projective limit of an inverse system of $C^{\ast}$-algebras
as follows: For each $\lambda \in \Lambda $, let $\mathcal{I}_{\lambda
}:=\lbrace a\in \mathcal{A}: p_{\lambda }(a)=0\rbrace $. Clearly, $\mathcal{I
}_{\lambda }$ is a closed two-sided $\ast $-ideal in $\mathcal{A}$ and $
\mathcal{A}_{\lambda }:=\mathcal{A}/\mathcal{I}_{\lambda }$ is a $C^{\ast}$
-algebra with respect to the $C^{\ast}$-norm $\left\Vert \cdot \right\Vert
_{\mathcal{A}_{\lambda }}$ induced by $p_{\lambda }$ (see \cite{A}). The
canonical quotient $\ast $-morphism from $\mathcal{A}$ to $\mathcal{A}
_{\lambda }$ is denoted by $\pi _{\lambda }^{\mathcal{A}}$. For each $
\lambda _{1},\lambda _{2}\in \Lambda $ with $\lambda _{1}\leq \lambda _{2}$,
there is a canonical surjective $\ast $-morphism $\pi _{\lambda _{2}\lambda
_{1}}^{\mathcal{A}}:\mathcal{A}_{\lambda _{2}}\rightarrow \mathcal{A}
_{\lambda _{1}}$ defined by $\pi _{\lambda _{2}\lambda _{1}}^{\mathcal{A}
}\left( a+\mathcal{I}_{\lambda _{2}}\right) =a+\mathcal{I}_{\lambda _{1}}$.
Then $\left\lbrace \mathcal{A}_{\lambda },\pi _{\lambda _{2}\lambda _{1}}^{
\mathcal{A}},\lambda _{1}\leq \lambda _{2},\ \lambda _{1},\lambda _{2}\in
\Lambda \right\rbrace $ forms an inverse system of $C^{\ast }$-algebras, because $
\pi _{\lambda _{1}}^{\mathcal{A}}=\pi _{\lambda _{2}\lambda _{1}}^{\mathcal{A
}}\circ \pi _{\lambda _{2}}^{\mathcal{A}}$ whenever $\lambda _{1}\leq
\lambda _{2}$. The projective limit
\begin{equation*}
\varprojlim\limits_{\lambda }\mathcal{A_{\lambda }}:=\left\lbrace \lbrace a_{\lambda
}\rbrace_{\lambda \in \Lambda }\in \prod_{\lambda \in \Lambda }\mathcal{
A_{\lambda }}: \pi _{\lambda _{2}\lambda _{1}}^{\mathcal{A}}(a_{\lambda
_{2}})=a_{\lambda _{1}}\text{ whenever}\ \lambda _{1}\leq \lambda
_{2}, \lambda _{1}, \lambda _{2}\in \Lambda \right\rbrace
\end{equation*}
of the inverse system of $C^{\ast}$-algebras $\left\lbrace \mathcal{A}_{\lambda
},\pi _{\lambda _{2}\lambda _{1}}^{\mathcal{A}},\lambda _{1}\leq \lambda
_{2},\ \lambda _{1},\lambda _{2}\in \Lambda \right\rbrace $ is a locally $C^{\ast
}$-algebra that may be identified with $\mathcal{A}$ by the map $a\mapsto
\left( \pi _{\lambda }^{\mathcal{A}}(a)\right) _{\lambda \in \Lambda }$. The
above relation is known as the Arens-Michael decomposition of $\mathcal{A}$
\cite[pg. 16]{Fr}$.$

Let $\mathcal{A}$ and $\mathcal{B}$ be two locally $C^{\ast}$-algebras
whose topologies are given by the families of $C^{\ast}$-seminorms $\left\lbrace
p_{\lambda }\right\rbrace _{\lambda \in \Lambda }$ and $\left\lbrace
q_{\delta }\right\rbrace_{\delta \in \Delta }$, respectively. A \textit{local contractive} $
\ast $\textit{-morphism} from $\mathcal{A}$ to $\mathcal{B}$ is a $\ast $-morphism $\pi :\mathcal{A}\rightarrow \mathcal{B}$ with the property that
for each $\delta \in \Delta ,$ there exists $\lambda \in \Lambda $ such that 
$q_{\delta }\left( \pi \left( a\right) \right) \leq p_{\lambda }\left(
a\right) $ for all $a\in \mathcal{A}$. If $\pi :\mathcal{A}\rightarrow 
\mathcal{B}$ is a $\ast $-morphism, $\Delta =\Lambda $ and $q_{\lambda
}\left( \pi \left( a\right) \right) = p_{\lambda }\left( a\right) $ for all $
a\in \mathcal{A}$ and $\lambda \in \Lambda ,$ we say that $\pi $ is a \textit{local isometric }$\ast $\textit{-morphism} from $\mathcal{A}$ to $\mathcal{B}$.

\begin{remark}
If $\pi :\mathcal{A}\rightarrow \mathcal{B}$ is a local isometric $\ast $-morphism, then $\pi \left( \mathcal{A}\right) $, the image of $\pi 
$, is a locally $C^{\ast}$-subalgebra of $\mathcal{B}$ and $\pi ^{-1}:\pi
\left( \mathcal{A}\right) \rightarrow \mathcal{A}$ is a local isometric $
\ast $\textit{-morphism.}
\end{remark}

\subsection{Positive and local positive elements}

Let $\mathcal{A}$ be a locally $C^{\ast}$-algebra whose topology is defined
by the family of $C^{\ast}$-seminorms $\lbrace p_{\lambda }\rbrace_{\lambda \in
\Lambda }$. An element $a\in \mathcal{A}$ is \textit{self-adjoint} if $
a=a^{\ast}$ and it is \textit{positive} if $a=b^{\ast}b$ for some $b\in 
\mathcal{A}.$\ 

An element $a\in \mathcal{A}$ is called \textit{local self-adjoint} if $
a=a^{\ast}+c$ for some $c\in \mathcal{A}$ with $p_{\lambda }\left( c\right)
=0$ for some $\lambda \in \Lambda ,$ and we call $a$ as $\lambda $
-self-adjoint, and \textit{local positive} if $a=b^{\ast}b+c$ where $b,c\in 
\mathcal{A}$ and $p_{\lambda }\left( c\right) =0$ for some $\lambda \in
\Lambda $; we call $a$ as $\lambda $-positive and write $a\geq _{\lambda }0$
. We write $a=_{\lambda }0$ whenever $p_{\lambda }\left( a\right) =0$.

\begin{remark}
An element $a\in \mathcal{A}$ is local self-adjoint if and only if there is $
\lambda \in \Lambda $ such that $\pi _{\lambda }^{\mathcal{A}}\left(
a\right) $ is self-adjoint in $\mathcal{A}_{\lambda }$ and $a\in \mathcal{A}$
is local positive if and only if there is $\lambda \in \Lambda $ such that $
\pi _{\lambda }^{\mathcal{A}}\left( a\right) $ is positive in $\mathcal{A}
_{\lambda }$.
\end{remark}

Note that an element $a\in \mathcal{A}$ is self-adjoint if and only if $a$ is 
$\lambda $-self-adjoint for all $\lambda \in \Lambda \ $and $a$ is positive
if and only if $a$ is $\lambda $-positive for all $\lambda \in \Lambda .$

\subsection{Local completely positive maps}

Let $\mathcal{A}$ and  $\mathcal{B}$ be two locally $C^{\ast}$-algebras
whose topologies are defined by the families of $C^{\ast}$-seminorms $
\lbrace p_{\lambda }\rbrace_{\lambda \in \Lambda }$ and $\lbrace q_{\delta
}\rbrace_{\delta \in \Delta }$, respectively. For each positive integer $n, \ M_{n}(\mathcal{A})$
denotes the collection of all matrices of order $n$ with elements in $
\mathcal{A}$. Note that $M_{n}(\mathcal{A})$ is a locally $C^{\ast}$
-algebra with respect to the family of $C^{\ast}$-seminorms $\lbrace p_{\lambda
}^{n}\rbrace_{\lambda \in \Lambda }$, where $p_{\lambda }^{n}\left(
[a_{ij}]_{i,j=1}^{n}\right) = \left\Vert [\pi _{\lambda }^{\mathcal{A}
}\left( a_{ij}\right) ]_{i,j=1}^{n}\right\Vert _{M_{n}(\mathcal{A}_{\lambda
})}$ for all $\lambda \in \Lambda .$

For each positive integer $n$, the $n$-amplification of a linear map $
\varphi :\mathcal{A}\rightarrow \mathcal{B}$ is the map $\varphi ^{\left(
n\right) }:M_{n}(\mathcal{A})\rightarrow M_{n}(\mathcal{B})$ defined by $
\varphi ^{\left( n\right) }\left( [a_{ij}]_{i,j=1}^{n}\right) =[\varphi
\left( a_{ij}\right) ]_{i,j=1}^{n}.$

A linear map $\varphi :\mathcal{A}\rightarrow \mathcal{B}$ is called 

\begin{enumerate}
\item \textit{positive} if $\varphi \left( a\right) \geq 0$ whenever $
a\geq 0$ for all $a\in \mathcal{A}.$

\item \textit{local positive} if for each $\delta \in \Delta $, there exists 
$\lambda \in \Lambda $ such that $\varphi \left( a\right) \geq _{\delta }0 $
whenever $a\geq _{\lambda }0$ and $\varphi \left( a\right) =_{\delta }0$
whenever $a=_{\lambda }0.$

\item \textit{completely positive} if $\varphi ^{\left( n\right) }\left( 
\left[ a_{ij}\right] _{i,j=1}^{n}\right) $\ $\geq 0$ whenever $\left[ a_{ij}
\right] _{i,j=1}^{n}\geq 0$ for all $n\geq 1$.

\item \textit{local completely positive (local $\mathcal{CP}$)} if
for each $\delta \in \Delta $, there exists $\lambda \in \Lambda $ such that 
$\varphi ^{\left( n\right) }\left( \left[ a_{ij}\right] _{i,j=1}^{n}\right) $
\ $\geq _{\delta }0\ $ whenever $\left[ a_{ij}\right] _{i,j=1}^{n}\geq
_{\lambda }0$ and $\varphi ^{\left( n\right) }\left( \left[ a_{ij}\right]
_{i,j=1}^{n}\right) =_{\delta }0$ whenever $\left[ a_{ij}\right]
_{i,j=1}^{n}=_{\lambda }0$ for all $n\geq 1$.

\item \textit{admissible local completely positive (admissible local $
\mathcal{CP}$)} if $\Delta =\Lambda $, and for each $\lambda \in
\Lambda ,$ $\varphi ^{\left( n\right) }\left( \left[ a_{ij}\right]
_{i,j=1}^{n}\right) $\ $\geq _{\lambda }0\ $ whenever $\left[ a_{ij}\right]
_{i,j=1}^{n}\geq _{\lambda }0$ and $\varphi ^{\left( n\right) }\left( \left[
a_{ij}\right] _{i,j=1}^{n}\right) =_{\lambda }0$ whenever $\left[ a_{ij}
\right] _{i,j=1}^{n}=_{\lambda }0$ for all $n\geq 1$.
\end{enumerate}

Note that any local contractive $\ast $-morphism $\pi :\mathcal{A}
\rightarrow \mathcal{B}$ is a local completely positive map. It is known
that the positivity property of the linear maps between $C^{\ast}$-algebras
implies their continuity. This property is not true in the case of positive
linear maps between locally $C^{\ast}$-algebras, but it is true in the case
of local positive linear maps \cite[Proposition 3.1]{J}. Note that a linear
map $\varphi :\mathcal{A}\rightarrow \mathcal{B}$ is local completely
positive if and only if it is continuous and completely positive \cite[
Proposition 3.3]{J}. If $:\mathcal{A}\rightarrow \mathcal{B}$ is local
completely positive, then $\varphi \left( a^{\ast }\right) =\varphi \left(
a\right) ^{\ast }$ for all $a\in \mathcal{A}$.

\begin{remark}\label{a}
If $\varphi :\mathcal{A}\rightarrow \mathcal{B}$ is an admissible
local $\mathcal{CP}$-map, then for each $\lambda \in \Lambda $, there exists
a $\mathcal{CP}$-map $\varphi _{\lambda }:\mathcal{A}_{\lambda }\rightarrow 
\mathcal{B}_{\lambda }$ such that $\varphi _{\lambda }\circ \pi _{\lambda }^{
\mathcal{A}}=\pi _{\lambda }^{\mathcal{B}}\circ \varphi $. Moreover, $\left(
\varphi _{\lambda }\right) _{\lambda \in \Lambda }$ is an inverse system of
completely positive maps and $\varphi =\varprojlim\limits_{\lambda}\varphi _{\lambda }.$
\end{remark}

\begin{lemma}\label{U}
Let $\varphi :\mathcal{A}\rightarrow \mathcal{B}$ be a bijective
unital linear map. If $\varphi $ and $\varphi ^{-1}$ are unital admissible
local $\mathcal{CP}$-maps, then  $\varphi $ is a local isometric $\ast $
-isomorphism.
\end{lemma}

\begin{proof}
Since $\varphi $ and $\varphi ^{-1}$ are unital admissible local $\mathcal{CP
}$-maps, for each $\lambda \in \Lambda $, there exist $\varphi _{\lambda }:
\mathcal{A}_{\lambda }\rightarrow \mathcal{B}_{\lambda }$ and $\left(
\varphi ^{-1}\right) _{\lambda }:\mathcal{B}_{\lambda }\rightarrow \mathcal{A
}_{\lambda }$ such that $\varphi _{\lambda }\circ \pi _{\lambda }^{\mathcal{A
}}=\pi _{\lambda }^{\mathcal{B}}\circ \varphi $ and $\left( \varphi
^{-1}\right) _{\lambda }\circ \pi _{\lambda }^{\mathcal{B}}=\pi _{\lambda }^{
\mathcal{A}}\circ \varphi ^{-1}$, respectively. Moreover, $\varphi _{\lambda
}$ and $\left( \varphi ^{-1}\right) _{\lambda }$ are unital isometric $
\mathcal{CP}$-maps, $\left( \varphi _{\lambda }\right) ^{-1}=\left( \varphi
^{-1}\right) _{\lambda }$, and by \cite[Lemma 2.7]{H}, $\varphi _{\lambda }$
is an isometric $\ast $-isomorphism. Then, for each $a,b\in \mathcal{A},$ we
have 
\begin{equation*}
\pi _{\lambda }^{\mathcal{B}}\left( \varphi \left( ab\right) -\varphi \left(
a\right) \varphi \left( b\right) \right) =\varphi _{\lambda }\left( \pi
_{\lambda }^{\mathcal{A}}\left( ab\right) \right) -\varphi _{\lambda }\left(
\pi _{\lambda }^{\mathcal{A}}\left( a\right) \right) \varphi _{\lambda
}\left( \pi _{\lambda }^{\mathcal{A}}\left( b\right) \right) =0
\end{equation*}
for all  $\lambda \in \Lambda $, and so $\varphi \left( ab\right) =\varphi
\left( a\right) \varphi \left( b\right) $. Therefore, $\varphi $ is a local
isometric $\ast $- isomorphism.
\end{proof}

The following theorem is a local convex version of \cite[Theorem 2.1]{G}
(see also \cite{CS}).

\begin{theorem}
Let $\varphi :\mathcal{A}\rightarrow \mathcal{B}$ be a unital local $
\mathcal{CP}$-map. Then we have

\begin{enumerate}
\item (Schwarz Inequality) $\varphi \left( a\right) ^{\ast }\varphi \left(
a\right) $\ $\leq \varphi \left( a^{\ast }a\right) $ for all $a\in \mathcal{A
}$ \cite[Corollary 5.5]{D}.

\item Let $a\in \mathcal{A}$. Then

\begin{enumerate}
\item $\varphi \left( a\right) ^{\ast }\varphi \left( a\right)=\varphi
\left( a^{\ast }a\right)$ if and only if $\varphi \left( ba\right) =
\varphi \left( b\right) \varphi \left( a\right) $ for all $b\in \mathcal{A}$
\cite[Corollary 5.5]{D}.

\item $\varphi \left( a\right) \varphi \left( a\right) ^{\ast } = \varphi
\left( aa^{\ast }\right) $ if and only if $\varphi \left( ab\right) =
\varphi \left( a\right) \varphi \left( b\right) $ for all $b\in \mathcal{A}
\ $\ \cite[Corollary 5.5]{D}.
\end{enumerate}

\item $\mathcal{M}_{\varphi }=\left\lbrace a\in \mathcal{A}: \varphi \left( a\right)
^{\ast }\varphi \left( a\right) =\varphi \left( a^{\ast }a\right) \ and \ 
\varphi \left( a\right) \varphi \left( a\right) ^{\ast } =\varphi \left(
aa^{\ast }\right) \right\rbrace$ is a unital locally $C^{\ast}$-subalgebra of $
\mathcal{A}$ and it is the largest locally $C^{\ast }$-subalgebra $\mathcal{C
}$ of $\mathcal{A}$ such that $\varphi\restriction_{\mathcal{C}}:
\mathcal{C}\rightarrow \mathcal{B}$ is a unital local contractive $\ast $-morphism.
Moreover, $\varphi \left( bac\right) =\varphi \left( b\right) \varphi \left(
a\right) \varphi \left( c\right) $ for all $b,c\in \mathcal{M}_{\varphi }$
and for all $a\in \mathcal{A}$.
\end{enumerate}
\end{theorem}

\begin{remark}\label{R1}
Let $a,b\in \mathcal{A}$. If $\varphi :\mathcal{A}\rightarrow 
\mathcal{B}$ is an admissible local $\mathcal{CP}$-map, then
\begin{equation*}
p_{\lambda }\left( b\right) ^{2}\varphi \left( a^{\ast }a\right) \geq
_{\lambda }\varphi \left( a^{\ast }b^{\ast }ba\right)
\end{equation*}
for all $\lambda \in \Lambda $, since $p_{\lambda }\left( b\right)
^{2}a^{\ast }a\geq _{\lambda }a^{\ast }b^{\ast }ba$ and $\varphi $ is an
admissible local $\mathcal{CP}$-map.
\end{remark}

\begin{definition}
A unital local $\mathcal{CP}$-map $\varphi :\mathcal{A}\rightarrow \mathcal{A
}$ is a \textit{projection} if $\varphi \circ \varphi =\varphi $. An \textit{admissible projection} on $\mathcal{A}$ is a unital admissible
local $\mathcal{CP}$-map $\varphi :\mathcal{A}\rightarrow \mathcal{A}$ such that $\varphi \circ \varphi =\varphi $.
\end{definition}

\begin{remark}
If $\varphi :\mathcal{A}\rightarrow \mathcal{A}$ is a projection, then $Im \left( \varphi \right) =\lbrace b\in \mathcal{A}: \left( \exists
\right) a\in \mathcal{A}$ such that $b=\varphi \left( a\right) \rbrace$ is a
closed subspace of $\mathcal{A}$.
\end{remark}

\begin{lemma}\label{L1}\cite[Corollary 5.6]{D}
 Let $\varphi :\mathcal{A}\rightarrow 
\mathcal{A}$ be a projection. Then
\begin{equation*}
\varphi \left( \varphi \left( a\right) \varphi \left( b\right) \right)
=\varphi \left( \varphi \left( a\right) b\right) =\varphi \left( a \varphi
\left( b\right) \right)
\end{equation*}
for all $a,b\in \mathcal{A}$.
\end{lemma}

As in the case of $C^{\ast}$-algebras, the range of an admissible projection on a unital locally $C^*$-algebra $\mathcal{A}$ has a structure of unital locally $C^*$-algebra.

\begin{proposition}\label{P1}
 Let $\varphi:\mathcal{A}\rightarrow \mathcal{A}$ be an admissible projection. Then $Im \left( \varphi \right) $ is a unital locally $C^{\ast}$-algebra with respect to the multiplication
defined by $b\cdot c=\varphi \left( bc\right) $ for all $b,c\in Im \left( \varphi \right) $, the involution induced by that on $\mathcal{A}$ 
and the family $\lbrace  p_{\lambda }\restriction_{Im\left( \varphi
\right) }\rbrace_{\lambda \in \Lambda }$ of  $C^{\ast}$-seminorms.
\end{proposition}

\begin{proof}
As in the case of $C^{\ast}$-algebras, using Lemma \ref{L1}, we obtain
that $Im\left( \varphi \right) $ has a structure of $\ast $
-algebra. From

\begin{enumerate}
\item $p_{\lambda }\left( b\cdot c\right) =p_{\lambda }\left( \varphi \left(
bc\right) \right) \leq p_{\lambda }\left( bc\right) \leq p_{\lambda }\left(
b\right) p_{\lambda }\left( c\right) $ for all $b,c\in Im\left(
\varphi \right) $ and for all $\lambda \in \Lambda ;$

\item $p_{\lambda }\left( b^{\ast }\right) =p_{\lambda }\left( b\right) $
for all $b\in Im\left( \varphi \right) $ and for all $\lambda \in
\Lambda ;$

\item $p_{\lambda }\left( b^{\ast }\cdot b\right) =p_{\lambda }\left(
\varphi \left( b^{\ast }b\right) \right) \leq p_{\lambda }\left( b^{\ast
}b\right) =p_{\lambda }\left( b\right) ^{2}=p_{\lambda }\left( \varphi
\left( b\right) \right) ^{2}=$

$p_{\lambda }\left( \varphi \left( b\right) ^{\ast }\varphi \left( b\right)
\right) \leq p_{\lambda }\left( \varphi \left( b^{\ast }b\right) \right)
=p_{\lambda }\left( b^{\ast }\cdot b\right) $ for all $b\in Im\left(
\varphi \right) $ and for all $\lambda \in \Lambda ,$
\end{enumerate}

we deduce that $\lbrace p_{\lambda}\restriction_{Im(\varphi)} \rbrace_{\lambda \in \Lambda }$ is a family of $C^{\ast}$-seminorms.
Therefore, $Im\left( \varphi \right) $ is a unital locally  $C^{\ast}$-algebra. 
\end{proof}

We point out that if $\varphi:\mathcal{A}\rightarrow \mathcal{A}$ is a
projection, then $Im\left( \varphi \right) $ is a $\ast $-algebra
with the multiplication and involution defined in Proposition \ref{P1}, but,
in general, $\lbrace p_{\lambda }\restriction_{Im\left( \varphi
\right) }\rbrace_{\lambda \in \Lambda }$ is not a family of $C^{\ast}$-seminorms.

The locally $C^{\ast}$-algebra $Im\left( \varphi \right) $ is denoted by $
C^{\ast }\left( \varphi \right) $. Let $j_{\varphi }:Im\left( \varphi
\right) \rightarrow C^{\ast }\left( \varphi \right)  $ be the canonical map.
 Clearly, $ j_{\varphi }$ is a surjective isometric linear map, and so,
there exists $j_{\varphi }^{-1}:C^{\ast }\left( \varphi \right) \rightarrow 
Im\left( \varphi \right) $ which is a surjective isometric linear
map. Moreover, $j_{\varphi }$ and $j_{\varphi }^{-1}$ are unital admissible
local $\mathcal{CP}$-maps.

\section{Admissible injective locally $C^{\ast}$-algebras}

Let $\mathcal{A}$ be a unital locally $C^{\ast}$-algebra with the topology
defined by the family of $C^{\ast}$-seminorms $\lbrace p_{\lambda }\rbrace_{\lambda
\in \Lambda }$. A linear subspace $\mathcal{S}$ of $\mathcal{A}$ is \textit{self-adjoint} if $\mathcal{S=S}^{\ast}$. An element $a$ in $
\mathcal{S}$ is local positive if it is local positive in $\mathcal{A}$.

A unital locally $C^{\ast}$-algebra $\mathcal{A}$ is \textit{injective} if
given any self-adjoint linear subspace $\mathcal{S}$ of a unital
locally $C^{\ast}$-algebra $\mathcal{B}$, containing the unit of $\mathcal{B}$, any unital local $\mathcal{CP}$- map from $\mathcal{S}$ to $\mathcal{A}$ extends to a unital local $\mathcal{CP}$-map from $\mathcal{B}$ to $\mathcal{A}$. By \cite[Theorem 8.1]{D}, locally $C^*$-algebra $C^*\left( \mathcal{D}_{\mathcal{E}} \right)$, where $\lbrace \mathcal{H}; \mathcal{E}=\lbrace \mathcal{H}_{n} \rbrace_{n\geq 1}; \mathcal{D}_{\mathcal{E}} \rbrace$ is a Fr\'{e}chet quantized domain, is an injective locally $C^*$-algebra.

\begin{definition}
A unital locally $C^*$-algebra $\mathcal{A}$ is admissible injective if for any self-adjoint subspace $\mathcal{S}$ containing the unit of a unital locally $C^*$-algebra $\mathcal{B}$, any unital admissible local $\mathcal{CP}$-map from $\mathcal{S}$ to $\mathcal{A}$ extends to a unital admissible local $\mathcal{CP}$-map from $\mathcal{B}$ to $\mathcal{A}$.

\end{definition}

\begin{remark}\label{Rem}
If $\lbrace \mathcal{H}; \mathcal{E}=\lbrace \mathcal{H}_{n} \rbrace_{n\geq 1}; \mathcal{D}_{\mathcal{E}} \rbrace$ is a Fr\'{e}chet quantized domain, then $C^*\left( \mathcal{D}_{\mathcal{E}} \right)$ is an admissible injective locally $C^*$-algebra (see, for example, the proof of \cite[Theorem 3.2]{J1}).
\end{remark}

\begin{remark}
Let $\mathcal{A}$ be a $C^*$-algebra. Then, $\mathcal{A}$ is admissible injective if and only if $\mathcal{A}$ is injective in the category whose objects are $C^*$-algebras and morphisms are completely positive maps.

\end{remark}

\begin{remark}
Let $\mathcal{A}$ be a unital Fr\'{e}chet locally $C^*$-algebra.  By \cite[Theorem 8.2]{D}, $\mathcal{A}$ can be identified with a locally $C^*$-subalgebra of 
$C^*\left( \mathcal{D}_{\mathcal{E}} \right)$ for a certain Fr\'{e}chet quantized domain $\lbrace \mathcal{H}; \mathcal{E}=\lbrace \mathcal{H}_{n} \rbrace_{n\geq 1}; \mathcal{D}_{\mathcal{E}} \rbrace$.
\begin{itemize}
\item[(1)] If $\mathcal{A}$ is admissible injective, then the identity map $id_{\mathcal{A}}:\mathcal{A}\rightarrow\mathcal{A}$ extends to an admissible projection $\varphi:C^*\left( \mathcal{D}_{\mathcal{E}} \right)\rightarrow C^*\left( \mathcal{D}_{\mathcal{E}} \right)$ whose range is $\mathcal{A}$, and by \cite[Theorem 8.2]{D}, $\mathcal{A}$ is injective.
\item[(2)] $\mathcal{A}$ is a unital admissible injective Fr\'{e}chet locally $C^*$-algebra if and only if it is the range of an admissible projection on $C^*\left( \mathcal{D}_{\mathcal{E}} \right)$.
\end{itemize}
\end{remark}

Let $\lbrace \mathcal{H}; \mathcal{E}=\lbrace \mathcal{H}_{n} \rbrace_{n\geq 1}; \mathcal{D}_{\mathcal{E}} \rbrace$ be a Fr\'{e}chet quantized domain. For each $n\geq 1$, $\mathcal{P}_{n}$ is the orthogonal projection of $\mathcal{H}$ on $\mathcal{H}_{n}$. Let $\mathcal{R}$ be the commutative subring in $B(\mathcal{H})$ generated by $\lbrace \mathcal{P}_{n} \rbrace_{n\geq 1}\cup\lbrace id_{\mathcal{H}} \rbrace$. A \textit{quantum $\mathcal{R}$-module projection} on $C^*\left( \mathcal{D}_{\mathcal{E}} \right)$ is a projection $\varphi:C^*\left( \mathcal{D}_{\mathcal{E}} \right)\rightarrow C^*\left( \mathcal{D}_{\mathcal{E}} \right)$ with the property that $\varphi(eT)=e\varphi(T)$ for all $T\in C^*\left( \mathcal{D}_{\mathcal{E}} \right)$ and for all $e\in\mathcal{R}$ \cite{DD}.

An injective locally $C^*$-algebra $\mathcal{A}\subseteq C^*\left( \mathcal{D}_{\mathcal{E}} \right)$ with the property that $\mathcal{RA}\subseteq\mathcal{A}$ is called an \textit{injective quantum $\mathcal{R}$-module} \cite[Section 4.1]{DD}. By \cite[Lemma 8.1]{D}, a locally $C^*$-algebra $\mathcal{A}\subseteq C^*\left( \mathcal{D}_{\mathcal{E}} \right)$ is an injective quantum $\mathcal{R}$-module if and only if it is the range of a quantum $\mathcal{R}$-module projection on $C^*\left( \mathcal{D}_{\mathcal{E}} \right)$.

\begin{remark}
If $\varphi:C^*\left( \mathcal{D}_{\mathcal{E}} \right)\rightarrow C^*\left( \mathcal{D}_{\mathcal{E}} \right)$ is a \textit{quantum $\mathcal{R}$-module projection} on $C^*\left( \mathcal{D}_{\mathcal{E}} \right)$, then $\varphi$ is an admissible projection on $C^*\left( \mathcal{D}_{\mathcal{E}} \right)$. Therefore, any unital injective quantum $\mathcal{R}$-module Fr\'{e}chet locally $C^*$-algebra $\mathcal{A}\subseteq C^*\left( \mathcal{D}_{\mathcal{E}} \right)$ is admissible injective.

\end{remark}

A \textit{Fr\'{e}chet $W^*$-algebra} is an inverse limit of a countable inverse system of $W^*$-algebras whose connecting maps are $W^*$-morphisms \cite{F1}.

\begin{remark}\label{Rem2}
Let $\mathcal{A}$ be a Fr\'{e}chet locally $W^*$-algebra. By Dosiev \cite[Theorem 3.1]{D1}, $\mathcal{A}$ can be identified with a locally $C^*$-subalgebra of $C^*\left( \mathcal{D}_{\mathcal{E}} \right)$ for some Fr\'{e}chet quantized domain  $\lbrace \mathcal{H}; \mathcal{E}=\lbrace \mathcal{H}_{n} \rbrace_{n\geq 1}; \mathcal{D}_{\mathcal{E}} \rbrace$ such that $\mathcal{P}_{n}\mathcal{A}\subseteq\mathcal{A}$ for all $n\geq 1$. Therefore, $\mathcal{A}$ is injective if and only if $\mathcal{A}$ is an injective quantum $\mathcal{R}$-module if and only if $\mathcal{A}$ is admissible injective.

\end{remark}

\begin{lemma}\label{projection}
 Let $\varphi :\mathcal{A}\rightarrow \mathcal{A}$ be an admissible projection. If $\mathcal{A}$ is admissible injective, then $C^{\ast }\left( \varphi \right) $ is admissible injective.
\end{lemma}

\begin{proof}
We seen that $j_{\varphi }$ and $j_{\varphi }^{-1}$ are unital admissible
local $\mathcal{CP}$-maps. Since $Im\left( \varphi \right) \subseteq 
\mathcal{A}$, we can assume that $j_{\varphi }^{-1}$ is a unital admissible
local $\mathcal{CP}$-map from $C^{\ast }\left( \varphi \right) $ to $
\mathcal{A}$. Let$\ \mathcal{B}$ be a unital locally $C^{\ast }$-algebra, $
\mathcal{S}$ be a self-adjoint subspace of $\mathcal{B}$ containing the unit
of $\mathcal{B}$ and $\psi :\mathcal{S}\rightarrow C^{\ast }\left( \varphi
\right) $ be a unital admissible local $\mathcal{CP}$-map. Then, $j_{\varphi
}^{-1}\circ \psi :\mathcal{S}\rightarrow \mathcal{A}$ is a unital admissible local $\mathcal{CP}$-map, and since $\mathcal{A}$ is admissible injective,
there exists a unital admissible local $\mathcal{CP}$-map $\widetilde{\psi }:
\mathcal{B}\rightarrow \mathcal{A}$ such that $ \widetilde{\psi }
\restriction_{\mathcal{S}}=j_{\varphi }^{-1}\circ \psi $. Let $j_{\varphi
}\circ \varphi \circ \widetilde{\psi }:\mathcal{B\rightarrow }C^{\ast
}\left( \varphi \right) $ be a map. Clearly, $j_{\varphi }\circ \varphi \circ 
\widetilde{\psi }$ is a unital admissible local $\mathcal{CP}$-map, and $
 j_{\varphi }\circ \varphi \circ \widetilde{\psi }\restriction_{
\mathcal{S}}=j_{\varphi }\circ \varphi \circ j_{\varphi }^{-1}\circ \psi
=\psi $. Therefore, $C^{\ast }\left( \varphi \right) $ is admissible injective.
\end{proof}

\section{Minimal projections on admissible injective  locally $C^{\ast}$-algebras}

\subsection{Minimal family of seminorms}

Let $\mathcal{A}$ be a unital locally $C^{\ast}$-algebra whose topology is
defined by the family of $C^{\ast}$-seminorms $\lbrace p_{\lambda }\rbrace_{\lambda
\in \Lambda }$ and $\mathcal{B\subseteq A}$ be a locally $C^{\ast}$
-subalgebra which contains the unit of $\mathcal{A}$.

\begin{definition}\label{m}
 A directed family of seminorms $\lbrace \widetilde{p}_{\lambda
}\rbrace_{\lambda \in \Lambda }$ on $\mathcal{A}$ is a family of $\mathcal{B}$-seminorms if for each $\lambda \in \Lambda $, the following conditions are satisfied:

\begin{enumerate}
\item $\widetilde{p}_{\lambda }\left( a\right) \leq p_{\lambda }\left(
a\right) $ for all $a\in \mathcal{A};$

\item $\widetilde{p}_{\lambda }\left( b\right) =p_{\lambda }\left( b\right) $
for all $b\in \mathcal{B};$

\item $\widetilde{p}_{\lambda }\left( bac\right) \leq p_{\lambda }\left(
b\right) \widetilde{p}_{\lambda }\left( a\right) p_{\lambda }\left( c\right) 
$ for all $a\in \mathcal{A}$ and for all $b,c\in \mathcal{B}$.
\end{enumerate}
\end{definition}

\begin{lemma}\label{help}
 Let $\varphi:\mathcal{A}\rightarrow \mathcal{A}$  be a unital admissible local $\mathcal{CP}$-map such that $\varphi \left( b\right) =b$
for all $b\in \mathcal{B}$. Then

\begin{enumerate}
\item $\left\lbrace p_{\lambda }\circ \varphi \right\rbrace_{\lambda \in \Lambda }$
is a family of $\mathcal{B}$-seminorms on $\mathcal{A}$.

\item $\left\lbrace \widetilde{p}_{\lambda}\right\rbrace _{\lambda \in \Lambda }$ is a family of $\mathcal{B}$-seminorms on $\mathcal{A}$,
where 
\begin{equation*}
\widetilde{p}_{\lambda }\left( a\right) =\limsup\limits_{n}\frac{1}{n}
p_{\lambda }\left( \varphi \left( a\right) +\varphi ^{2}\left( a\right)
+\cdot \cdot \cdot +\varphi ^{n}\left( a\right) \right)
\end{equation*}
\end{enumerate}
\end{lemma}

\begin{proof}
$\left( 1\right) $ By the admissibility of $\varphi $, we have 
\begin{equation*}
\left( p_{\lambda }\circ \varphi \right) \left( a\right) \leq p_{\lambda
}\left( \varphi \left( a\right) \right) \leq p_{\lambda }\left( a\right)
\end{equation*}
for all $a\in \mathcal{A}$ and for all $\lambda \in \Lambda $. Since $
\varphi \left( b\right) =b$ for all $b\in \mathcal{B}$, we have
\begin{equation*}
\left( p_{\lambda }\circ \varphi \right) \left( b\right) =p_{\lambda }\left(
b\right)
\end{equation*}
for all $b\in \mathcal{B}$ and for all $\lambda \in \Lambda $. By \cite[Corrolary 5.5]{D}, $\varphi \left( bac\right) =\varphi \left( b\right)
\varphi \left( a\right) \varphi \left( c\right) =b\varphi \left( a\right) c$
for all $a\in \mathcal{A}$ and for all $b,c\in \mathcal{B}$, since $\varphi
\left( b^{\ast }b\right) =\varphi \left( b\right) ^{\ast }\varphi \left(
b\right) $ and $\varphi \left( bb^{\ast }\right) =\varphi \left( b\right)
\varphi \left( b\right) ^{\ast }$ for all $b\in \mathcal{B}$. Then 
\begin{equation*}
\left( \left( p_{\lambda }\circ \varphi \right) \left( bac\right) \right)
=p_{\lambda }\left( b\varphi \left( a\right) c\right) \leq p_{\lambda
}\left( b\right) \left( p_{\lambda }\circ \varphi \right) \left( a\right)
p_{\lambda }\left( c\right)
\end{equation*}
for all $a\in \mathcal{A}$ and for all $b,c\in \mathcal{B}$ and for all $
\lambda \in \Lambda $.

$\left( 2\right)$ We have $\widetilde{p}_{\lambda }\left( a\right) \leq
p_{\lambda }\left( \varphi \left( a\right) \right) \leq p_{\lambda }\left(
a\right) $ for all $a\in \mathcal{A}$ and for all $\lambda \in \Lambda $, 
and $\widetilde{p}_{\lambda }\left( b\right) =p_{\lambda }\left( b\right) $
for all $b\in \mathcal{B}$ and for all $\lambda \in \Lambda $. Also 
\begin{eqnarray*}
\widetilde{p}_{\lambda }\left( bac\right) &=&\limsup_{n}\frac{1}{n}
p_{\lambda }\left( \varphi \left( bac\right) +\cdot \cdot \cdot +\varphi
^{n}\left( bac\right) \right) \\
&=&\limsup_{n}\frac{1}{n}p_{\lambda }\left( b\left( \varphi \left( a\right)
+\cdot \cdot \cdot +\varphi ^{n}\left( a\right) \right) c\right) \\
&\leq &p_{\lambda }\left( b\right) \limsup_{n}\frac{1}{n}p_{\lambda }\left(
\varphi \left( a\right) +\cdot \cdot \cdot +\varphi ^{n}\left( a\right)
\right) p_{\lambda }\left( c\right) \\
&=&p_{\lambda }\left( b\right) \widetilde{p}_{\lambda }\left( a\right)
p_{\lambda }\left( c\right)
\end{eqnarray*}
for all $a\in \mathcal{A}$ and for all $b,c \in \mathcal{B}$ and for all $\lambda \in \Lambda $.
\end{proof}

\begin{remark}\label{1}
If $\lbrace \widetilde{p}_{\lambda }\rbrace_{\lambda \in \Lambda }$ is a
family of $\mathcal{B}$-seminorms on $\mathcal{A}$, then for each $\lambda
\in \Lambda $, there exists a $\mathcal{B}_{\lambda }$-seminorm $\widetilde{p}_{\lambda ,\mathcal{A}_{\lambda }}$ on $\mathcal{A}_{\lambda }$ such that $\widetilde{p}_{\lambda ,\mathcal{A}_{\lambda }}\left( \pi _{\lambda }^{\mathcal{A}}\left( a\right) \right) =\widetilde{p}_{\lambda }\left( a\right) 
$ for all $a\in \mathcal{A}$.
 Indeed, for each $\lambda \in \Lambda ,$ since 
$\widetilde{p}_{\lambda }\left( a\right) \leq p_{\lambda }\left( a\right) $
for all $a\in \mathcal{A}$, there is a map $\widetilde{p}_{\lambda ,\mathcal{A}_{\lambda }}:\mathcal{A}_{\lambda }\rightarrow [ 0,\infty )$ such that $\widetilde{p}_{\lambda ,\mathcal{A}_{\lambda }}\left( \pi _{\lambda }^{\mathcal{A}}\left( a\right) \right) =\widetilde{p}_{\lambda }\left( a\right) 
$ for all $a\in \mathcal{A}$. Clearly, $\widetilde{p}_{\lambda ,\mathcal{A}_{\lambda }}$ is a seminorm. On the other hand, we have

\begin{enumerate}
\item $\widetilde{p}_{\lambda ,\mathcal{A}_{\lambda }}\left( \pi _{\lambda
}^{\mathcal{A}}\left( a\right) \right) =\widetilde{p}_{\lambda }\left(
a\right) \leq p_{\lambda }\left( a\right) =\left\Vert \pi _{\lambda }^{\mathcal{A}}\left( a\right) \right\Vert _{\mathcal{A}_{\lambda }}$, for all $a\in \mathcal{A};$

\item $\widetilde{p}_{\lambda ,\mathcal{A}_{\lambda }}\left( \pi _{\lambda
}^{\mathcal{A}}\left( b\right) \right) =\widetilde{p}_{\lambda }\left(
b\right) =p_{\lambda }\left( b\right) =\left\Vert \pi _{\lambda }^{\mathcal{A}}\left( b\right) \right\Vert _{\mathcal{A}_{\lambda }}$, for all $b\in 
\mathcal{B}$;

 \item
\begin{align*}
\widetilde{p}_{\lambda ,\mathcal{A}_{\lambda }}\left( \pi _{\lambda}^{\mathcal{A}}\left( b\right) \pi _{\lambda }^{\mathcal{A}}\left( a\right) \pi _{\lambda }^{\mathcal{A}}\left( c\right) \right) 
&= \widetilde{p}_{\lambda,\mathcal{A}_{\lambda }}\left( \pi _{\lambda }^{\mathcal{A}}\left( bac\right) \right) \\
&= \widetilde{p}_{\lambda }\left( bac\right) \\
&\leq p_{\lambda }\left( b\right) \widetilde{p}_{\lambda }\left( a\right) p_{\lambda }\left( c\right) \\
&= \left\Vert \pi _{\lambda }^{\mathcal{A}}\left( b\right) \right\Vert _{\mathcal{A}_{\lambda }} \widetilde{p}_{\lambda ,\mathcal{A}_{\lambda }}\left( \pi _{\lambda }^{\mathcal{A}}\left( a\right) \right) \left\Vert \pi _{\lambda }^{\mathcal{A}}\left( c\right) \right\Vert _{\mathcal{A}_{\lambda }}.
\end{align*}
\end{enumerate}



Therefore, $\widetilde{p}_{\lambda ,\mathcal{A}_{\lambda }}$ is a $\mathcal{B}_{\lambda }$-seminorm on $\mathcal{A}_{\lambda }$.
\end{remark}

\begin{definition}
Let $\lbrace \widetilde{p}_{\lambda }\rbrace_{\lambda \in \Lambda }$ and $\lbrace \widetilde{q
}_{\lambda }\rbrace_{\lambda \in \Lambda }$ be two families of $\mathcal{B}$-seminorms on $\mathcal{A}$. We say that $\lbrace \widetilde{p}_{\lambda
}\rbrace_{\lambda \in \Lambda }\leq \lbrace \widetilde{q}_{\lambda }\rbrace_{\lambda \in
\Lambda}$ if for each $\lambda \in \Lambda $, $\widetilde{p}_{\lambda
}\left( a\right) \leq \widetilde{q}_{\lambda }\left( a\right) $ for all $a\in \mathcal{A}$.
\end{definition}

This relation is a partial ordering relation on the collection of all
families of $\mathcal{B}$-seminorms on $\mathcal{A}$.

\begin{proposition}\label{4}
Let $\mathcal{A}$ be a unital Fr\'{e}chet locally $C^{\ast}$-algebra whose topology is defined by the family of $C^{\ast}$-seminorms $\lbrace p_{n}\rbrace_{n \geq 1}$ and $\mathcal{B\subseteq A}$ be a locally $C^{\ast}$-subalgebra which contains the unit of $\mathcal{A}$. Then there exists a
minimal family of $\mathcal{B}$-seminorms on $\mathcal{A}$.
\end{proposition}

\begin{proof}
Let $n\geq 1$ and $p_{n,\mathcal{A}_{n}}^{\min}$ be a minimal $\mathcal{B}
_{n}$-seminorm on $\mathcal{A}_{n}$ \cite[p.187]{H}. Then, the map $\widetilde{p}_{n}^{\min }:\mathcal{A}\rightarrow [ 0,\infty )$ defined
by $\widetilde{p}_{n}^{\min }\left( a\right) =p_{n,\mathcal{A}_{n}}^{\min
}\left( \pi _{n}^{\mathcal{A}}\left( a\right) \right) $ is a seminorm on $
\mathcal{A}$. Moreover, $\widetilde{p}_{n}^{\min }$ verifies the conditions $
(1),(2)$ and $(3)$ from Definition \ref{m}. Consider the map $p_{n}^{\min }:\mathcal{A}\rightarrow [ 0,\infty )$ defined by
\begin{equation*}
p_{n}^{\min }\left( a\right) =\max \lbrace \widetilde{p}_{m}^{\min }\left(
a\right) : m\leq n \rbrace.
\end{equation*}
It is easy to verify that $p_{n}^{\min }$ is a seminorm on $\mathcal{A}$
that verifies the conditions $(1),(2)$ and $(3)$ from Definition \ref{m}.
Then $\lbrace p_{n}^{\min }\rbrace_{n \geq 1}$ is a family of $\mathcal{B}$-seminorms on 
$\mathcal{A}$.

Let $\lbrace q_{n}\rbrace_{n\geq 1}$ be another family of $\mathcal{B}$-seminorms on $
\mathcal{A}$ such that $\lbrace q_{n}\rbrace_{n\geq 1}\leq \lbrace p_{n}^{\min }\rbrace_{n\geq 1}$. By Remark \ref{1}, for each $m\geq 1$, there exists a $\mathcal{B}_{m}$-seminorm $q_{m,\mathcal{A}_{m}}$ on $\mathcal{A}_{m}$ such that $q_{m, \mathcal{A}_{m}}\left( \pi _{m}^{\mathcal{A}}\left( a\right) \right)
=q_{m}\left( a\right) $ for all $a\in \mathcal{A}$. Since $p_{1,\mathcal{A}
_{1}}^{\min }$ is a minimal $\mathcal{B}_{1}$-seminorm on $\mathcal{A}_{1},$
from 
\begin{equation*}
q_{1,\mathcal{A}_{1}}\left( \pi _{1}^{\mathcal{A}}\left( a\right) \right)
=q_{1}\left( a\right) \leq p_{1}^{\min }\left( a\right) =\widetilde{p}
_{1}^{\min }\left( a\right) =p_{1,\mathcal{A}_{1}}^{\min }\left( \pi _{1}^{\mathcal{A}}\left( a\right) \right)
\end{equation*}
for all $a\in \mathcal{A}$, we deduce that 
\begin{equation*}
q_{1}\left( a\right) =q_{1,\mathcal{A}_{1}}\left( \pi _{1}^{\mathcal{A}
}\left( a\right) \right) =p_{1,\mathcal{A}_{1}}^{\min }\left( \pi _{1}^{
\mathcal{A}}\left( a\right) \right) =p_{1}^{\min }\left( a\right) =
\widetilde{p}_{1}^{\min }\left( a\right)
\end{equation*}
for all $a\in \mathcal{A}$, and so $q_{1}=p_{1}^{\min }=\widetilde{p}
_{1}^{\min }$. From
\begin{equation*}
\widetilde{p}_{1}^{\min }\left( a\right) =q_{1}\left( a\right) \leq
q_{2}\left( a\right) \leq p_{2}^{\min }\left( a\right) =\max \lbrace \widetilde{p}
_{1}^{\min }\left( a\right) ,\widetilde{p}_{2}^{\min }\left( a\right) \rbrace
\end{equation*}
for all $a\in \mathcal{A}$, it follows that 
\begin{equation*}
q_{2}\left( a\right) \leq p_{2}^{\min }\left( a\right) =\widetilde{p}
_{2}^{\min }\left( a\right)
\end{equation*}
for all $a\in \mathcal{A}$. Therefore, $q_{2,\mathcal{A}_{2}}\left( \pi
_{2}^{\mathcal{A}}\left( a\right) \right) \leq p_{2,\mathcal{A}_{2}}^{\min
}\left( \pi _{2}^{\mathcal{A}}\left( a\right) \right) $ for all $a\in 
\mathcal{A}$, and by the minimality of $p_{2,\mathcal{A}_{2}}^{\min }$, we
obtain that $q_{2,\mathcal{A}_{2}}=p_{2,\mathcal{A}_{2}}^{\min }$.
Consequently, 
\begin{equation*}
q_{2}=\widetilde{p}_{2}^{\min }=p_{2}^{\min }.
\end{equation*}
By induction, we have 
\begin{equation*}
q_{n}=p_{n}^{\min }=\widetilde{p}_{n}^{\min }
\end{equation*}
for every $n\geq 1$. Therefore, $\lbrace p_{n}^{\min }\rbrace_{n\geq 1}$ is a minimal
family of $\mathcal{B}$-seminorms on $\mathcal{A}.$
\end{proof}

In the following lines, we recall some results about positive linear functionals on locally $C^{\ast}$-algebras and continuous $\ast $-representations on Hilbert
spaces. We refer the reader to \cite{Fr} for further information about
continuous $\ast $-representations of locally $C^{\ast}$-algebras.
\medskip

 Let $\mathcal{A}$ be a unital locally $C^{\ast}$-algebra whose topology is
defined by the family of $C^{\ast}$-seminorms $\lbrace p_{\lambda }\rbrace_{\lambda
\in \Lambda }$.
A linear functional $f$ on $\mathcal{A}$ is \textit{positive} if $f\left(
a^{\ast }a\right) \geq 0$ for all $a\in \mathcal{A}$, and it is $\lambda $
\textit{-positive} if $f\left( a\right) \geq 0$ whenever $a\geq _{\lambda }0$
and $f\left( a\right) =0$ whenever $a=_{\lambda }0$.

\begin{remark}\label{22}
 Let $f:\mathcal{A} \rightarrow \mathbb{C}$ be a linear functional.

\begin{enumerate}
\item $f$ is $\lambda$-positive for some $\lambda \in \Lambda $ if and only
if $f$ is continuous and positive.

\item $f$ is $\lambda$-positive for some $\lambda \in \Lambda $ if and only
if there exists a positive linear functional $f_{\lambda}:\mathcal{A}
_{\lambda}\rightarrow \mathbb{C}$ such that $f=f_{\lambda} \circ
\pi_{\lambda}^{\mathcal{A}}$.
\end{enumerate}
\end{remark}

A continuous unital positive linear functional on $\mathcal{A}$ is called a 
\textit{state} on $\mathcal{A}$. A $\lambda $\textit{-state} on $\mathcal{A}$
is a unital $\lambda $-positive linear functional on $\mathcal{A}$.

\begin{remark}\label{222}
 Let $f:\mathcal{A}\rightarrow \mathbb{C}$ be a linear functional.

\begin{enumerate}
\item $f$ is a state on $\mathcal{A}$ if and only if there exists $\lambda
\in \Lambda $ such that $f$ is a $\lambda $-state.

\item $f$ is a $\lambda $-state on $\mathcal{A}$ for some $\lambda \in
\Lambda$ if and only if there exists a state $f_{\lambda }$ on $\mathcal{A}
_{\lambda }$ such that $f=f_{\lambda }\circ \pi _{\lambda }^{\mathcal{A}}$.
\end{enumerate}
\end{remark}

A state $f:\mathcal{A}\rightarrow \mathbb{C}$ is \textit{pure} if whenever $
g:\mathcal{A}\rightarrow \mathbb{C}$ is a positive linear functional such
that $f-g$ is a positive linear functional on $\mathcal{A}$, there
exists $\alpha \in  [0,1]$ such that $g=\alpha f$.

A $\lambda $-state on $\mathcal{A}$ is \textit{pure} if whenever $g:\mathcal{A}
\rightarrow \mathbb{C}$ is a $\lambda $-positive linear functional such that 
$f-g$ is a $\lambda $-positive linear functional on $\mathcal{A}$, there
exists $\alpha \in  [0,1]$ such that $g=\alpha f$.

\begin{remark}\label{2}
Let $f:\mathcal{A}\rightarrow \mathbb{C}$ be a state. Then

\begin{enumerate}
\item $f$ is a pure state on $\mathcal{A}$ if and only if $f$ is a pure $\lambda $-state for some $\lambda \in \Lambda$.

\item $f$ is a pure $\lambda $-state on $\mathcal{A}$ if and only if there
exists a pure state $f_{\lambda }$ on $\mathcal{A}_{\lambda }$ such that $
f=f_{\lambda }\circ \pi _{\lambda }^{\mathcal{A}}$.
\end{enumerate}
\end{remark}

A continuous $\ast $-representation of $\mathcal{A}$ on a Hilbert space $\mathcal{H}$ is a continuous $\ast $-morphism $\pi :\mathcal{A}\rightarrow B(
\mathcal{H})$. A continuous $*$-representation $\pi:\mathcal{A}\rightarrow B(\mathcal{H})$ is irreducible if the only closed subspaces of $\mathcal{H}$ invariant under $\pi(\mathcal{A})$ are trivial subspaces $\lbrace 0 \rbrace$ and $\mathcal{H}$ itself.

Let $f:\mathcal{A}\rightarrow \mathbb{C}$ be a $\lambda $-state on $\mathcal{A}$ and $f_{\lambda }:\mathcal{A}_{\lambda }\rightarrow \mathbb{C}$ be a
state on $\mathcal{A}_{\lambda }$ such that $f=f_{\lambda }\circ \pi
_{\lambda }^{\mathcal{A}}$ (Remark \ref{22} (2)). By $GNS$ construction \cite[Section 3.4]{MG}, there exists a cyclic $\ast $-representation $(\pi
_{f_{\lambda }},\mathcal{H}_{f},\xi _{f})\ $ of $\mathcal{A}_{\lambda }$
such that $$f_{\lambda }\left( \pi _{\lambda }^{\mathcal{A}}\left( a\right)
\right) =\left\langle \pi _{f_{\lambda }}\left( \pi _{\lambda }^{\mathcal{A}
}\left( a\right) \right) \xi _{f},\xi _{f}\right\rangle ,$$ for all $a\in 
\mathcal{A}$. Then $\pi _{f}=\pi _{f_{\lambda }}\circ \pi _{\lambda }^{\mathcal{A}}$ is a continuous $\ast $-representation of $\mathcal{A}$ on $\mathcal{H}_{f}$. Therefore, there exists a continuous cyclic $\ast$-representation $\left( \pi _{f},\mathcal{H}_{f},\xi _{f}\right) $ such that 
$f\left( a\right) =\left\langle \pi _{f}\left( a\right) \xi _{f},\xi
_{f}\right\rangle $ for all $a\in \mathcal{A}$. Moreover, if $f$ is pure,
then $f_{\lambda }$ is pure and so, the cyclic $\ast $-representation $(\pi
_{f_{\lambda }},\mathcal{H}_{f},\xi _{f})\ $associated to $f_{\lambda }$ is
irreducible. Consequently, the continuous cyclic $\ast $-representation $(\pi _{f},\mathcal{H}_{f},\xi _{f})$ associated to $f$ is irreducible.

\begin{lemma}\label{mm}
Let $\mathcal{B\subseteq A}$ be a locally $C^{\ast }$-subalgebra of $\mathcal{A}$ and let $p$ be a seminorm on $\mathcal{A}$ such that $p\left(
a\right) \leq p_{\lambda }\left( a\right) $ for all $a\in \mathcal{A}$ and $p\left( b\right) =p_{\lambda }\left( b\right) $ for all $b\in \mathcal{B}$
and for some $\lambda \in \Lambda $. If $f$ is a pure $\lambda $-state on $\mathcal{B}$, then it extends to a $\lambda $-state $\widetilde{f}$ on $\mathcal{A}$ such that 
\begin{equation*}
\left\vert \widetilde{f}\left( a\right) \right\vert \leq p\left( a\right), \ \text{for all 
}a\in \mathcal{A}.
\end{equation*}
\end{lemma}

\begin{proof}
Since $f$ is a pure $\lambda $-state on $\mathcal{B}$, by Remark \ref{2}
(2), there exists a pure state $f_{\lambda }$ on $\mathcal{B}_{\lambda }$
such that $f= f_{\lambda }\circ \pi_{\lambda}^{\mathcal{A}}\restriction_{\mathcal{B}}$. On the other hand, since $p$ is a seminorm on $\mathcal{A}$ such that $p\left( a\right) \leq p_{\lambda }\left( a\right) $
for all $a\in \mathcal{A}$, by Remark \ref{1}, there exists a seminorm $p_{
\mathcal{A}_{\lambda }}$ on $\mathcal{A}_{\lambda }$ such that $p=$ $p_{\mathcal{A}_{\lambda }}\circ \pi _{\lambda }^{\mathcal{A}}$. Since 
\begin{equation*}
\left\vert f_{\lambda }\left( \pi _{\lambda }^{\mathcal{A}}\left( b\right)
\right) \right\vert \leq p\left( b\right) =p_{\mathcal{A}_{\lambda }}\left(
\pi _{\lambda }^{\mathcal{A}}\left( b\right) \right),
\end{equation*}
for all $b\in \mathcal{B}$, by Hahn-Banach theorem, there exists a state $\widetilde{f_{\lambda }}$ on $\mathcal{A}_{\lambda }$ such that $ \widetilde{f_{\lambda }}\restriction_{\mathcal{B}_{\lambda }}=f_{\lambda} $ and 
\begin{equation*}
\left\vert \widetilde{f_{\lambda }}\left( \pi _{\lambda }^{\mathcal{A}
}\left( a\right) \right) \right\vert \leq p_{\mathcal{A}_{\lambda }}\left(
\pi _{\lambda }^{\mathcal{A}}\left( a\right) \right),
\end{equation*}
for all $a\in \mathcal{A}$. Since 
\begin{equation*}
p_{\mathcal{A}_{\lambda }}\left( \pi _{\lambda }^{\mathcal{A}}\left(
a\right) \right) =p\left( a\right) \leq p_{\lambda }\left( a\right)=\Vert \pi_{\lambda}^{\mathcal{A}}(a) \Vert_{\mathcal{A}_{\lambda}},
\end{equation*}
for all $a\in \mathcal{A}$, there exists a $\lambda $-state $\widetilde{f}=
\widetilde{f_{\lambda }}\circ \pi _{\lambda }^{\mathcal{A}}$ on $\mathcal{A}$
such that $ \widetilde{f}\restriction_{\mathcal{B}}=f_{\lambda }$ and $\left\vert \widetilde{f}\left( a\right) \right\vert \leq p\left( a\right), $
for all $a\in \mathcal{A}$.
\end{proof}

For a local contractive $\ast $-morphism $\pi:\mathcal{A\rightarrow }
C^{\ast }(\mathcal{D}_{\mathcal{E}})$, we set $$\pi (\mathcal{B})^{\prime
}:=\left\lbrace T\in B(\mathcal{H}): T\pi \left( a\right) =\pi \left( a \right) T\restriction_{\mathcal{D}_{\mathcal{E}}}, (\forall) \ a\in \mathcal{A}\right\rbrace.$$

The following proposition plays a crucial role in showing the existence of a
minimal admissible $\mathcal{B}$-projection on an admissible injective
locally $C^{\ast}$-algebra.

\begin{proposition}\label{5}
Let $\mathcal{A}$ be a unital locally $C^{\ast}$-algebra whose topology is defined by the family of $C^{\ast}$-seminorms $\lbrace p_{\lambda
}\rbrace_{\lambda \in \Lambda }$, $\mathcal{B\subseteq A}$ be a locally $C^{\ast
} $-subalgebra which contains the unit of $\mathcal{A}$ and $\lbrace \widetilde{p}_{\lambda }\rbrace_{\lambda \in \Lambda }$ be a family of $\mathcal{B}$-seminorms
on $\mathcal{A}$. Then, there exist a quantized domain $\lbrace \mathcal{H},
\mathcal{E}=\lbrace \mathcal{H}_{\lambda} \rbrace_{\lambda \in \Lambda },\mathcal{D}_{
\mathcal{E}}\rbrace$, a quantized subdomain $\lbrace \mathcal{K},\mathcal{F}=\lbrace \mathcal{K}_{\lambda} \rbrace_{\lambda \in \Lambda },\mathcal{D}_{\mathcal{F}}\rbrace$ and a
unital local contractive $\ast $-morphism $\pi:\mathcal{A\rightarrow }
C^{\ast }(\mathcal{D}_{\mathcal{E}})$ such that

\begin{enumerate}
\item $\left\Vert \pi \left( b\right) \right\Vert _{\lambda }=\widetilde{p}
_{\lambda }\left( b\right) $, for all $b\in \mathcal{B}$ and for each $
\lambda \in \Lambda $;

\item $\left[ \pi \left( \mathcal{B}\right) \mathcal{H}_{\lambda }\right] =
\mathcal{K}_{\lambda }$, where $\left[ \pi \left( \mathcal{B}\right) 
\mathcal{H}_{\lambda }\right] $ is the closed subspace of $\mathcal{H}
_{\lambda }$ generated by $\lbrace \pi \left( b\right) \xi : b\in \mathcal{B}, \xi
\in \mathcal{H}_{\lambda }\rbrace$, for each $\lambda \in \Lambda .$
\end{enumerate}

Moreover, if $\mathcal{Q}$ is the projection of $\mathcal{H}$ onto $\mathcal{K}
$, then $\mathcal{Q}\restriction_{\mathcal{D}_{\mathcal{E}}}\in \pi (
\mathcal{B})^{\prime }\cap C^{\ast }(\mathcal{D}_{\mathcal{E}})$ and
\begin{equation*}
\left\Vert \mathcal{Q}\pi \left( a\right)  \mathcal{Q}\restriction_{
\mathcal{D}_{\mathcal{E}}}\right\Vert _{\lambda }\leq \widetilde{p}_{\lambda
}\left( a\right),
\end{equation*}
for all $a\in \mathcal{A}$ and for all $\lambda \in \Lambda $.
\end{proposition}

\begin{proof}
Let $\lambda \in \Lambda $ and $\mathfrak{s}_{\lambda }$ be the set of all
pure $\lambda $-states on $\mathcal{B}$. Let $f\in \mathfrak{s}_{\lambda }$.$
\ $Since $\widetilde{p}_{\lambda }$ is a seminorm on $\mathcal{A}$ such that 
$\widetilde{p}_{\lambda }\left( a\right) \leq p_{\lambda }\left( a\right) $
for all $a\in \mathcal{A}$, and $\widetilde{p}_{\lambda }\left( b\right)
=p_{\lambda }\left( b\right) $ for all $b\in \mathcal{B}$, by Lemma \ref{mm}
, $f$ extends to a $\lambda $-state $\widetilde{f}$ on $\mathcal{A}$ such
that $\left\vert \widetilde{f}\left( a\right) \right\vert \leq \widetilde{p}
_{\lambda }\left( a\right) $ for all $a\in \mathcal{A}$. Let $\left( \pi _{
\widetilde{f}},\mathcal{H}_{\widetilde{f}},\xi _{\widetilde{f}}\right) $ be
the continuous cyclic $\ast $-representation of $\mathcal{A}$ associated to $
\widetilde{f}$ and $\mathcal{K}_{f}=[\pi _{\widetilde{f}}(\mathcal{B})\xi _{
\widetilde{f}}]$, the closed subspace of $\mathcal{H}_{\widetilde{f}}$
generated by $\lbrace \pi _{\widetilde{f}}(b)\xi _{\widetilde{f}} : b\in \mathcal{B}
\rbrace$. Then $\left( \pi_{\widetilde{f}}\restriction_{\mathcal{B}}, K_{\widetilde{f}}, \xi _{\widetilde{f}}\right) $ is a continuous cyclic $\ast $-representation of $\mathcal{B}$ associated to $f$, and since $f$ is a pure $\lambda $-state on $\mathcal{B}$, $ \pi_{\widetilde{f}}\restriction_{\mathcal{B}}$ is irreducible.

Set $H_{\lambda }=\bigoplus\limits_{f\in \mathfrak{s}_{\lambda }}\mathcal{H}
_{\widetilde{f}},$ $K_{\lambda }=\bigoplus\limits_{f\in \mathfrak{s}
_{\lambda }}\mathcal{K}_{f}$ and let $Q_{\lambda }$ be the orthogonal
projection of $H_{\lambda }$ on $K_{\lambda }$. Then, the linear map $
\widetilde{\pi }_{\lambda }:\mathcal{A}\rightarrow B(H_{\lambda })$ defined
by
\begin{equation*}
\widetilde{\pi }_{\lambda }(a)\left( \left( \eta _{\widetilde{f}}\right)
_{f\in \mathfrak{s}_{\lambda }}\right) =\left( \pi _{\widetilde{f}}\left(
a\right) \eta _{\widetilde{f}}\right) _{f\in \mathfrak{s}_{\lambda }}
\end{equation*}
is a continuous $\ast $-representation of $\mathcal{A}\ $and $\left[ \pi
_{\lambda }\left( \mathcal{B}\right) H_{\lambda }\right] =K_{\lambda }$.
Moreover, 
\begin{equation*}
\left\Vert \widetilde{\pi }_{\lambda }(a)\right\Vert \leq p_{\lambda }\left(
a\right),
\end{equation*}
for all $a\in \mathcal{A}$, and since $ \widetilde{\pi }_{\lambda
}\restriction_{\mathcal{B}}$ is the universal representation of $\mathcal{B}$, 
\begin{equation*}
\left\Vert \widetilde{\pi }_{\lambda }(b)\right\Vert =p_{\lambda }\left(
b\right) \ =\widetilde{p}_{\lambda }\left( b\right),
\end{equation*}
for all $b\in \mathcal{B}$. Clearly, $Q_{\lambda }\in \widetilde{\pi }
_{\lambda }(\mathcal{B})^{\prime }$.

Let $a\in \mathcal{A}$ and $\left( \pi _{\widetilde{f}}(b_{f})\xi _{
\widetilde{f}}\right) _{f\in \mathfrak{s}_{\lambda }}\in K_{\lambda }$.
Since $\widetilde{\pi }_{\lambda }(\mathcal{B})$ acts irreducibly on $
K_{f},f\in \mathfrak{s}_{\lambda }$, we may assume that $\left\Vert \pi _{
\widetilde{f}}(b_{f})(\xi _{f})\right\Vert =p_{\lambda }(b_{f})$, for all $
f\in \mathfrak{s}_{\lambda }$. Then 

\begin{align*}
&\left\vert \left\langle Q_{\lambda }\widetilde{\pi }_{\lambda}(a)Q_{\lambda }\left( \left( \pi _{\widetilde{f}}(b_{f})\xi _{\widetilde{f}}\right) _{f\in \mathfrak{s}_{\lambda }}\right) ,\left( \pi _{\widetilde{f}}(b_{f})\xi _{\widetilde{f}}\right) _{f\in \mathfrak{s}_{\lambda }}\right\rangle \right\vert= \\
&= \left\vert \left\langle \left( \pi _{\widetilde{f}}(ab_{f})\xi _{\widetilde{f}}\right) _{f\in \mathfrak{s}_{\lambda }},\left( \pi _{\widetilde{f}}(b_{f})\xi _{\widetilde{f}}\right) _{f\in \mathfrak{s}_{\lambda }}\right\rangle \right\vert \\
&= \left\vert \left\langle \left( \pi _{\widetilde{f}}(b_{f}^{\ast}ab_{f})\xi _{\widetilde{f}}\right) _{f\in \mathfrak{s}_{\lambda }},\left( \xi _{\widetilde{f}}\right) _{f\in \mathfrak{s}_{\lambda }}\right\rangle \right\vert \\
&\leq \sum\limits_{f\in \mathfrak{s}_{\lambda }}\left\vert \widetilde{f}\left( b_{f}^{\ast }ab_{f}\right) \right\vert\leq \sum\limits_{f\in \mathfrak{s}_{\lambda }}\widetilde{p}_{\lambda }\left( b_{f}^{\ast}ab_{f}\right)\leq \widetilde{p}_{\lambda }\left( a\right) \sum\limits_{f\in \mathfrak{s}_{\lambda }}p_{\lambda }\left( b_{f}\right) ^{2} \\
&= \widetilde{p}_{\lambda }\left( a\right) \sum\limits_{f\in \mathfrak{s}_{\lambda }}\left\Vert \pi _{\widetilde{f}}\left( b_{f}\right) (\xi _{f})\right\Vert ^{2} = \widetilde{p}_{\lambda }\left( a\right) \left\Vert \left( \pi _{\widetilde{f}}(b_{f})\xi _{\widetilde{f}}\right) _{f\in \mathfrak{s}_{\lambda }}\right\Vert ^{2}. 
\end{align*}


Therefore, 
\begin{equation*}
\left\Vert Q_{\lambda }\widetilde{\pi }_{\lambda }\left( a\right) Q_{\lambda
}\right\Vert \leq \widetilde{p}_{\lambda }\left( a\right), (\forall) \ a\in \mathcal{A}.
\end{equation*}

Let $\mathcal{H=}\bigoplus\limits_{\lambda \in \Lambda }H_{\lambda }$ and $
\mathcal{K=}\bigoplus\limits_{\lambda \in \Lambda }K_{\lambda }$. Then $\lbrace
\mathcal{H}, \mathcal{E} =\lbrace \mathcal{H}_{\lambda} \rbrace_{\lambda \in \Lambda },\mathcal{D}_{\mathcal{E}}\rbrace$, where $\mathcal{H}_{\lambda }=\bigoplus\limits_{\mu \leq \lambda }H_{\mu }$, is a quantized
domain in $\mathcal{H}$ and $\lbrace \mathcal{K}, \mathcal{F} =\lbrace \mathcal{K}_{\lambda} \rbrace_{\lambda \in \Lambda }, \mathcal{D}_{
\mathcal{F}}\rbrace$, where $\mathcal{K}_{\lambda }=\bigoplus\limits_{\mu \leq
\lambda }K_{\mu }$, is a quantized domain in $\mathcal{K}$. Moreover, $\lbrace
\mathcal{K}, \mathcal{F} =\lbrace \mathcal{K}_{\lambda}  \rbrace_{\lambda \in \Lambda },\mathcal{D}_{\mathcal{F}}\rbrace$ is a quantized
subdomain in $\lbrace \mathcal{H}, \mathcal{E} =\lbrace \mathcal{H}_{\lambda}  \rbrace_{\lambda \in \Lambda },\mathcal{D}_{\mathcal{E}}\rbrace$.
For each $\lambda \in \Lambda $, consider the linear map $\pi _{\lambda }:\mathcal{A}\rightarrow B(\mathcal{H}_{\lambda })$ defined by 
\begin{equation*}
\pi _{\lambda }\left( a\right) \left( \left( \xi _{\mu }\right) _{\mu \leq
\lambda }\right) =\left( \widetilde{\pi }_{\mu }\left( a\right) \xi _{\mu
}\right) _{\mu \leq \lambda }.
\end{equation*}
Clearly, $\pi _{\lambda }$ is a $\ast $-morphism. Moreover, 
\begin{equation*}
\left\Vert \pi _{\lambda }\left( a\right) \right\Vert =\sup \lbrace \left\Vert 
\widetilde{\pi }_{\mu }\left( a\right) \right\Vert : \mu \leq \lambda \rbrace \leq
\sup \lbrace p_{\mu }\left( a\right) : \mu \leq \lambda \rbrace =p_{\lambda }\left(
a\right)
\end{equation*}
for all $a\in \mathcal{A}$, 
\begin{equation*}
\left\Vert \pi _{\lambda }\left( b\right) \right\Vert =\sup \lbrace \left\Vert 
\widetilde{\pi }_{\mu }\left( b\right) \right\Vert : \mu \leq \lambda \rbrace =\sup
\lbrace p_{\mu }\left( b\right) : \mu \leq \lambda \rbrace =p_{\lambda }\left( b\right) =
\widetilde{p}_{\lambda }\left( b\right)
\end{equation*}
for each $b\in \mathcal{B}$,$\ $and $\left[ \pi _{\lambda }\left( \mathcal{B}
\right) \mathcal{H}_{\lambda }\right] =\mathcal{K}_{\lambda }$. If $\mathcal{Q}_{\lambda }$ is the projection of $\mathcal{H}_{\lambda }$ on $\mathcal{K}
_{\lambda },$ then $\mathcal{Q}_{\lambda }\in \pi _{\lambda }(\mathcal{B}
)^{\prime }$ and 
\begin{eqnarray*}
\left\Vert \mathcal{Q}_{\lambda }\pi _{\lambda }\left( a\right) \mathcal{Q}
_{\lambda }\right\Vert &=&\sup \left\lbrace \left\Vert Q_{\mu }\pi _{\mu }\left(
a\right) Q_{\mu }\right\Vert :\mu \leq \lambda \right\rbrace \\
&\leq &\sup \left\lbrace \widetilde{p}_{\mu }\left( a\right) : \mu \leq \lambda
\right\rbrace =\widetilde{p}_{\lambda }\left( a\right),
\end{eqnarray*}
for all $a\in \mathcal{A}$.

Let $a\in \mathcal{A}$. Consider the linear map $\pi \left( a\right):\mathcal{D}_{\mathcal{E}}\rightarrow \mathcal{D}_{\mathcal{E}}$ defined by 
\begin{equation*}
\pi \left( a\right) \xi =\pi _{\lambda }\left( a\right) \xi \text{ if }\xi
\in \mathcal{H}_{\lambda }\text{.}
\end{equation*}
It is easy to check that it is well-defined, $\pi \left( a\right) \in
C^{\ast }(\mathcal{D}_{\mathcal{E}})$ and $\pi \left( a\right) ^{\ast }=\pi
\left( a^{\ast }\right) $. In this way, we obtain a $\ast $-morphism $\pi: \mathcal{A}\rightarrow C^{\ast }(\mathcal{D}_{\mathcal{E}})$. Moreover,
for each $\lambda \in \Lambda $,
\begin{equation*}
\left\Vert \pi \left( a\right) \right\Vert _{\lambda }=\left\Vert \pi
_{\lambda }\left( a\right) \right\Vert \leq p_{\lambda }\left( a\right),
\end{equation*}
for all $a\in \mathcal{A}$ and 
\begin{equation*}
\left\Vert \pi \left( b\right) \right\Vert _{\lambda }=\left\Vert \pi
_{\lambda }\left( b\right) \right\Vert =p_{\lambda }\left( b\right) =
\widetilde{p}_{\lambda }\left( b\right),
\end{equation*}
for all $b\in \mathcal{B}$. Consequently, $\pi $ is a local contractive $
\ast $-morphism and $\left[ \pi \left( \mathcal{B}\right) \mathcal{H}
_{\lambda }\right] =\left[ \pi _{\lambda }\left( \mathcal{B}\right) \mathcal{H}_{\lambda }\right] =\mathcal{K}_{\lambda }$, for all $\lambda \in \Lambda $.

Let $\mathcal{Q}$ be the projection of $\mathcal{H}$ on $\mathcal{K}$.
Clearly, for each $\lambda \in \Lambda , \mathcal{Q}\restriction_{\mathcal{H}_{\lambda }}=\mathcal{Q}_{\lambda }$, and so $ \mathcal{Q}
\restriction_{\mathcal{D}_{\mathcal{E}}}\in C^{\ast }(\mathcal{D}_{\mathcal{E
}})\cap \pi (\mathcal{B})^{\prime }$. Then, for each $\lambda \in \Lambda 
$ and for each $a\in \mathcal{A}$, we have 
\begin{equation*}
\left\Vert \mathcal{Q}\pi \left( a\right)  \mathcal{Q}\restriction_{\mathcal{D}_{\mathcal{E}}}\right\Vert _{\lambda }=\left\Vert \mathcal{Q}
_{\lambda }\pi _{\lambda }\left( a\right) \mathcal{Q}_{\lambda }\right\Vert
\leq \widetilde{p}_{\lambda }\left( a\right).
\end{equation*}
\end{proof}

\subsection{Minimal projections}

Let $\mathcal{A}$ be a unital locally $C^{\ast}$-algebra whose topology is
defined by the family of $C^{\ast}$-seminorms $\lbrace p_{\lambda }\rbrace_{\lambda
\in \Lambda }$ and let $\mathcal{B\subseteq A}$ be a locally $C^{\ast}$-subalgebra which contains the unit of $\mathcal{A}$.

\begin{definition}
A linear map $\varphi:\mathcal{A}\rightarrow \mathcal{A}$ is a $\mathcal{B}$-projection (respectively, admissible $\mathcal{B}$-projection) if it is a projection (respectively, an admissible projection) and $\varphi \left(
b\right) =b$, for all $b\in \mathcal{B}$.
\end{definition}

\begin{definition}
If $\varphi $ and $\psi $ are two $\mathcal{B}$-projections (respectively,
admissible $\mathcal{B}$-projections) on $\mathcal{A}$, we say that $\varphi
\prec \psi $ if $\psi \circ \varphi =\varphi \circ \psi =\varphi $.
\end{definition}

This relation is a partial ordering relation on the collection of all $
\mathcal{B}$-projections (respectively, admissible $\mathcal{B}$-projections) on $\mathcal{A}$.

\begin{theorem}\label{Help}
Let $\mathcal{A}$ be a unital Fr\'{e}chet locally $C^{\ast}$-algebra whose topology is defined by the family of $C^{\ast}$-seminorms $
\lbrace p_{n}\rbrace_{n\geq 1}$ and $\mathcal{B\subseteq A}$ be a locally $C^{\ast }$
-subalgebra which contains the unit of $\mathcal{A}$. If $\mathcal{A}$ is admissible injective, then there exists a minimal admissible $\mathcal{B}$-projection on $\mathcal{A}$.
\end{theorem}

\begin{proof}
By Proposition \ref{4}, there exists a minimal family of $\mathcal{B}$
-seminorms $\lbrace p_{n}^{\min }\rbrace_{n\geq 1}$ on $\mathcal{A}$, and by
Proposition \ref{5}, there exist a quantized domain $\lbrace \mathcal{H}, \mathcal{E}=\lbrace \mathcal{H}_{n} \rbrace_{n\geq 1}, \mathcal{D}_{\mathcal{E}}\rbrace$, a quantized
subdomain $\lbrace \mathcal{K}, \mathcal{F}= \lbrace  \mathcal{K}_{n} \rbrace_{n \geq 1},\mathcal{D}_{\mathcal{F}}\rbrace$ and a local contractive $\ast $-morphism $\pi:\mathcal{A\rightarrow }C^{\ast }(\mathcal{D}_{\mathcal{E}})$ such that for each $
n\geq 1$, 
\begin{equation*}
\left\Vert \pi \left( b\right) \right\Vert _{n}=p_{n}^{\min }\left( b\right),
\end{equation*}
for all $b\in \mathcal{B},$ and  
\begin{equation*}
\left\Vert \mathcal{Q}\pi \left( a\right)  \mathcal{Q}\restriction_{
\mathcal{D}_{\mathcal{E}}}\right\Vert _{n}\leq p_{n}^{\min }\left( a\right),
\end{equation*}
for all $a\in \mathcal{A}$, where $\mathcal{Q}$ is the projection of $
\mathcal{H}$ on $\mathcal{K}$ and $\left[ \pi \left( \mathcal{B}\right) 
\mathcal{H}_{n}\right] =\mathcal{K}_{n}$. Moreover, $ \mathcal{Q}
\restriction_{\mathcal{D}_{\mathcal{E}}}\in \pi \left( \mathcal{B}\right)
^{\prime }\cap C^{\ast }(\mathcal{D}_{\mathcal{E}})$ and $\mathcal{Q}\pi
\left( \mathcal{B}\right) $ is a unital locally $C^{\ast }$-subalgebra of $
C^{\ast }(\mathcal{D}_{\mathcal{E}})$. Since, for each $n\geq 1$,
\begin{equation*}
\left\Vert \mathcal{Q}\pi \left( b\right) \right\Vert _{n}=\left\Vert \pi
\left( b\right) \right\Vert _{n}=p_{n}^{\min }\left( b\right) =p_{n}\left(
b\right),
\end{equation*}
for all $b\in \mathcal{B}$, the map $\Phi:\mathcal{Q}\pi \left( \mathcal{B}
\right) \rightarrow \mathcal{B}$ defined by $\Phi \left( \mathcal{Q}\pi
\left( b\right) \right) =b$ is a unital local isometric $\ast $-isomorphism
and so, it is a unital admissible local $\mathcal{CP}$-map. We can regard $
\Phi $ as a unital admissible local $\mathcal{CP}$-map from $\mathcal{Q}\pi
\left( \mathcal{B}\right) $ to $\mathcal{A}$. On the other hand, $\mathcal{Q}
\pi \left( \mathcal{B}\right) $ and $\mathcal{Q}\pi \left( a\right) \mathcal{Q}$ can be identified with self-adjoint subspaces of the unital Fr\'{e}chet
locally $C^{\ast }$-algebra $C^{\ast }(\mathcal{D}_{\mathcal{F}})$
containing the unit of $C^{\ast }(\mathcal{D}_{\mathcal{F}}).$ Therefore,
since $\mathcal{A}$ is injective in the category of unital Fr\'{e}chet
locally $C^{\ast }$-algebras and unital admissible local $\mathcal{CP}$-maps as morphisms , $\Phi $ extends to a unital admissible local $\mathcal{CP}$
-map $\widetilde{\Phi }$ from $C^{\ast }(\mathcal{D}_{\mathcal{F}})$ to $
\mathcal{A}$. Let $\varphi:\mathcal{A\rightarrow A}$ be given by 
\begin{equation*}
\varphi \left( a\right) =\widetilde{\Phi }\left(  \mathcal{Q}\pi
\left( a\right) \mathcal{Q}\restriction_{\mathcal{D}_{\mathcal{E}}}\right).
\end{equation*}
Clearly, $\varphi $ is a unital admissible local $\mathcal{CP}$-map, and for
each $b\in \mathcal{B}$,
\begin{equation*}
\varphi \left( b\right) =\Phi \left( \mathcal{Q\pi }\left( b\right) \right)
=b\text{.}
\end{equation*}
Moreover, for each $n\geq 1$,
\begin{equation*}
p_{n}\left( \varphi \left( a\right) \right) =p_{n}\left( \widetilde{\Phi }
\left( \mathcal{Q}\pi \left( a\right)  \mathcal{Q}\restriction_{
\mathcal{D}_{\mathcal{E}}}\right) \right) \leq \left\Vert \mathcal{Q}\pi
\left( a\right) \mathcal{Q}\restriction_{\mathcal{D}_{\mathcal{E}
}}\right\Vert _{n}\leq p_{n}^{\min }\left( a\right),
\end{equation*}
for all $a\in \mathcal{A}$.

To show that $\varphi $ is a $\mathcal{B}$-projection on $\mathcal{A}$, it
remains to prove that $\varphi \circ \varphi =\varphi $. By Remark \ref{help}, $\left\lbrace  p_{n}\circ \varphi \right\rbrace_{n\geq 1}$ and $\lbrace \widetilde{p}
_{n}\rbrace_{n\geq 1}$ are families of $\mathcal{B}$-seminorms on $\mathcal{A}$, where 
\begin{equation*}
\widetilde{p}_{n}\left( a\right) =\limsup\limits_{m}\frac{1}{m}p_{n}\left(
\varphi \left( a\right) +\cdot \cdot \cdot +\varphi ^{m}\left( a\right)
\right) , a\in \mathcal{A}.
\end{equation*}
Moreover, $\left\lbrace p_{n}\circ \varphi \right\rbrace_{n\geq 1}\leq \lbrace p_{n}^{\min
}\rbrace_{n\geq 1}$ and $\left\lbrace \widetilde{p}_{n}\right\rbrace_{n\geq 1}\leq
\lbrace p_{n}^{\min }\rbrace_{n\geq 1}$, and then, by the minimality of $\lbrace p_{n}^{\min
}\rbrace_{n\geq 1}$, we have 
\begin{equation*}
p_{n}\circ \varphi =p_{n}^{\min }=\widetilde{p}_{n}, (\forall)\ n\geq 1.
\end{equation*}
 Therefore,
\begin{eqnarray*}
p_{n}\left( \varphi \left( a\right) -\varphi ^{2}\left( a\right) \right)
&=&\left( p_{n}\circ \varphi \right) \left( a-\varphi \left( a\right)
\right) =\widetilde{p}_{n}\left( a-\varphi \left( a\right) \right) \\
&=&\limsup\limits_{m}\frac{1}{m}p_{n}\left( \varphi \left( a\right) -\varphi
^{m+1}\left( a\right) \right) \leq \lim\limits_{m}\frac{2}{m}p_{n}\left(
a\right) =0,
\end{eqnarray*}
for all $a\in \mathcal{A}$ and for all $n\geq 1$. Consequently, $\varphi
=\varphi \circ \varphi $, and so $\varphi $ is an admissible $\mathcal{B}$
-projection on $\mathcal{A}$.

To show the minimality of $\varphi $, let $\psi $ be another admissible $
\mathcal{B}$-projection on $\mathcal{A}$ such that $\psi \prec \varphi $.
Then $\psi \circ \varphi =\varphi \circ \psi =\psi $. By Remark \ref{help}, $
\left\lbrace p_{n}\circ \varphi \right\rbrace _{n\geq 1}$ and $\left\lbrace p_{n}\circ \psi
\right\rbrace _{n\geq 1}$ are families of $\mathcal{B}$-seminorms on $
\mathcal{A}$, and since 
\begin{equation*}
\left( p_{n}\circ \psi \right) \left( a\right) =p_{n}\left( \psi \left(
\varphi \left( a\right) \right) \right) \leq p_{n}\left( \varphi \left(
a\right) \right) =\left( p_{n}\circ \varphi \right) \left( a\right) \leq
p_{n}^{\min }\left( a\right),
\end{equation*}
for all $a\in \mathcal{A}$ and for all $n\geq 1$, and  by the minimality of $
\lbrace p_{n}^{\min }\rbrace _{n\geq 1}$, we deduce that $p_{n}\circ \psi =p_{n}\circ
\varphi $ for all $n\geq 1$. Consequently, $\ker \psi =\ker \varphi $. Let  
$a\in \mathcal{A}$. Since $\psi \left( \psi \left( a\right) -\varphi \left(
a\right) \right) =0$, we deduce that $\psi \left( a\right) -\varphi \left(
a\right) \in \ker \psi $. Therefore, $\psi \left( a\right) -\varphi \left(
a\right) \in \ker \varphi ,$ and then 
\begin{equation*}
0=\varphi \left( \psi \left( a\right) -\varphi \left( a\right) \right) =\psi
\left( a\right) -\varphi \left( a\right) .
\end{equation*}
Consequently, $\psi =\varphi $.
\end{proof}

\begin{remark}\label{Help1}
As in the case of $C^{\ast}$-algebras \cite[Remark 3.6]{H},
we obtain a surjective map from the set of all minimal admissible $\mathcal{B}$-projections on $\mathcal{A}$ to the set of all minimal families
of $\mathcal{B}$-seminorms on $\mathcal{A}$, $\varphi \mapsto \left\lbrace
p_{n}\circ \varphi \right\rbrace _{n\geq 1}$.

 Indeed, if $\varphi $ is a
minimal admissible $\mathcal{B}$-projection on $\mathcal{A}$, then by
Lemma \ref{help}, $\left\lbrace p_{n}\circ \varphi \right\rbrace _{n\geq 1}$ is a
family of $\mathcal{B}$-seminorms on $\mathcal{A}$. 
 By the proof of Theorem \ref{Help}, if $\lbrace q_{n} \rbrace_{n\geq 1}$ is a family of $\mathcal{B}$-seminorms on $\mathcal{A}$
 such that $\lbrace q_{n}\rbrace _{n\geq 1}\leq \left\lbrace p_{n}\circ \varphi \right\rbrace _{n\geq 1},$ there exists an admissible $\mathcal{B}$-projection $\psi $ on $\mathcal{A}$
such that $q_{n}=p_{n}\circ \psi $, for all $n\geq 1.$ Since $\lbrace p_{n}\circ
\psi \rbrace _{n\geq 1}\leq \left\lbrace p_{n}\circ \varphi \right\rbrace _{n\geq 1}$, it
follows that $\ker \varphi \subseteq \ker \psi $. From this relation and
taking into account that $\varphi $ is a projection, we deduce that $\psi
=\psi \circ \varphi $. Then $\varphi \circ $ $\psi $ is an admissible $
\mathcal{B}$-projection on $\mathcal{A}$. Moreover, $\varphi \circ \psi $ $
\prec \varphi $ and then, by the minimality of $\varphi $, we have $\varphi
=\varphi \circ \psi $. Since
\begin{equation*}
\left( p_{n}\circ \psi \right) \left( a\right) \leq \left( p_{n}\circ
\varphi \right) \left( a\right) =p_{n}\left( \varphi \left( \psi \left(
a\right) \right) \right) \leq p_{n}\left( \psi \left( a\right) \right)
=\left( p_{n}\circ \psi \right) \left( a\right),
\end{equation*}
for all $a\in \mathcal{A},$ it follows that $p_{n}\circ \psi
=p_{n}\circ \varphi .$ Therefore, $\left\lbrace p_{n}\circ \varphi \right\rbrace
_{n \geq 1 }$ is a minimal family of $\mathcal{B}$-seminorms, and
the map $\varphi \mapsto \left\lbrace p_{n}\circ \varphi \right\rbrace _{n \geq 1 }$ from the set of all minimal admissible $\mathcal{B}$-projections on $\mathcal{A}$ to the set of all minimal families of $\mathcal{B}$-seminorms on $\mathcal{A}$ is well-defined. The surjectivity follows from the proof of Theorem \ref{Help}.
\end{remark}

\section{Admissible injective envelope for a Fr\'{e}chet locally $C^{\ast }$-algebra}

Let $\mathcal{A}$ be a unital locally $C^{\ast}$-algebra whose topology is
defined by the family of $C^{\ast}$-seminorms $\lbrace p_{\lambda }\rbrace _{\lambda
\in \Lambda }$.

\begin{definition}
A pair $(\mathcal{B},\phi )$ of a unital locally $C^{\ast }$-algebra $\mathcal{B}$ and a unital local isometric $\ast $-morphism $\phi:\mathcal{A}
\rightarrow \mathcal{B}$ is called an admissible extension of $\mathcal{A}$.
An admissible extension $(\mathcal{B},\phi )$ of $\mathcal{A}$ is admissible injective if the locally $C^{\ast }$-algebra $\mathcal{B}$ is admissible injective.
\end{definition}

\begin{remark}\label{i}
Any unital Fr\'{e}chet locally $C^{\ast }$-algebra $\mathcal{A}$
has an admissible injective extension. Indeed, by \cite[Theorem 7.2]{D},
there exist a Fr\'{e}chet quantized domain $\lbrace \mathcal{H}; \mathcal{E} = \lbrace \mathcal{H}_{n} \rbrace _{n\geq 1}; \mathcal{D}_{\mathcal{E}}\rbrace$ and a unital local
isometric $\ast $-morphism $\pi:\mathcal{A}\rightarrow C^{\ast }(\mathcal{D}_{\mathcal{E}})$. By  Remark \ref{Rem}, $C^{\ast }(\mathcal{D}_{\mathcal{E}})$ is admissible injective.
\end{remark}

\begin{remark}
A unital Fr\'{e}chet locally $C^*$-algebra $\mathcal{A}\subseteq C^{\ast }(\mathcal{D}_{\mathcal{E}})$ can be identified with a locally $C^*$-subalgebra of its injective $\mathcal{R}$-envelope $\mathcal{I}_{\mathcal{R}}(\mathcal{A})$ \cite{DD}. Then $\left( \mathcal{I}_{\mathcal{R}}(\mathcal{A}), i \right)$ is an admissible injective extension of $\mathcal{A}$.

\end{remark}

\begin{definition}
An admissible injective envelope of $\mathcal{A}$ is an admissible extension $(\mathcal{B},\phi )$ of $\mathcal{A}$ with the property that id$_{\mathcal{B
}}$ is the unique unital admissible local $\mathcal{CP}$-map from $\mathcal{B}$ to $\mathcal{B}$ which fixes each element in $\phi (\mathcal{A})$.
\end{definition}

The following theorem is the main result of this paper.

\begin{theorem}\label{p}
 Any unital Fr\'{e}chet locally $C^{\ast}$-algebra $\mathcal{A}$
has an admissible injective envelope $(\mathcal{B},\phi )$, which is unique
in the sense that if $(\mathcal{B}_{1},\phi _{1})$ is another admissible
injective envelope for $\mathcal{A}$, then there exists a unique unital
local isometric $\ast $-isomorphism $\Phi:\mathcal{B}\rightarrow \mathcal{B}_{1}$ such that $\Phi \circ \phi =\phi _{1}$.
\end{theorem}

To prove this theorem, we need first the following two lemmas.

\begin{lemma}\label{p1}
Let $\mathcal{A}$ be a unital admissible injective  Fr\'{e}chet locally $C^{\ast}$-algebra, $\mathcal{B} $ a locally $C^{\ast }$-subalgebra and $\varphi $ a minimal admissible $\mathcal{B}$-projection on $\mathcal{A}$. Then id$_{C^{\ast }\left( \varphi \right) }$
is the unique unital admissible local completely positive map from $C^{\ast
}\left( \varphi \right) $ to $C^{\ast }\left( \varphi \right) $ which
extends id$_{\mathcal{B}}.$
\end{lemma}

\begin{proof}
Let $\psi :C^{\ast }\left( \varphi \right) \rightarrow $ $C^{\ast }\left(
\varphi \right) $ be a unital admissible local $\mathcal{CP}$-map such that $ \psi \restriction _{\mathcal{B}}= id_{\mathcal{B}}$. By Remark 
\ref{Help1}, $\left\lbrace  p_{n}\circ \varphi \right\rbrace _{n\geq 1}\ $ is a minimal
family of $\mathcal{B}$-seminorms on $\mathcal{A}$.  If $\left\lbrace 
\widehat{p_{n}}\right\rbrace _{n\geq 1}$ is a family of $\mathcal{B}$-seminorms
on $C^{\ast }\left( \varphi \right) $, then $\left\lbrace \widehat{p_{n}}\circ
\varphi \right\rbrace _{n\geq 1}$ is a family of $\mathcal{B}$-seminorms on $
\mathcal{A}$. Moreover, $\left\lbrace \widehat{p_{n}}\circ \varphi \right\rbrace
_{n\geq 1}\leq \left\lbrace p_{n}\circ \varphi \right\rbrace _{n\geq 1}$, and by the
minimality of $\left\lbrace p_{n}\circ \varphi \right\rbrace _{n\geq 1}$, we deduce
that $ \widehat{p_{n}}= p_{n}\restriction _{C^{\ast }\left( \varphi
\right) }$ for all $n\geq 1$. On the other hand, by Remark \ref{help}, $
\left\lbrace \widetilde{p_{n}}\right\rbrace _{n\geq 1}$ is a family of $\mathcal{B}$-seminorms on $C^{\ast }\left( \varphi
\right) $, where 
\begin{equation*}
\widetilde{p_{n}}\left( a\right) =\limsup\limits_{m}\frac{1}{m}p_{n}\left(
\psi \left( a\right) +\cdot \cdot \cdot +\psi ^{m}\left( a\right) \right)
,a\in C^{\ast }\left( \varphi \right),
\end{equation*}
 and so $\left\lbrace \widetilde{p_{n}}\right\rbrace _{n\geq 1}=\left\lbrace
 p_{n}\restriction _{C^{\ast}\left( \varphi \right) }\right\rbrace _{n\geq
1}$. Then
\begin{equation*}
p_{n}\left( a-\psi \left( a\right) \right) =\widetilde{p_{n}}\left( a-\psi
\left( a\right) \right) =\limsup\limits_{m}\frac{1}{m}p_{n}\left( \psi
\left( a\right) -\psi ^{m+1}\left( a\right) \right) =0,
\end{equation*}
for all $a\in C^{\ast }\left( \varphi \right)$ and for all $n\geq 1$. Consequently, $\psi =id_{C^{\ast }\left( \varphi \right) }$.
\end{proof}

The above lemma is a local convex version of \cite[Lemma 3.7]{H} and the
following lemma is a local convex version of \cite[Lemma 3.8]{H}.

\begin{lemma}\label{p2}
Let $\mathcal{A}_{1}$ and $\mathcal{A}_{2}$ be two unital admissible injective Fr\'{e}chet locally $C^{\ast }$-algebras whose topologies
are given by the families of $C^{\ast }$-seminorms $\left\lbrace p_{n}\right\rbrace
_{n\geq 1}$ and $\left\lbrace q_{n}\right\rbrace _{n\geq 1}$, respectively. Let $\mathcal{B}
_{1}\ $ be a locally $C^{\ast }$-subalgebra of $\mathcal{A}_{1}$, $\mathcal{B}
_{2}\ $ a locally $C^{\ast }$-subalgebra of $\mathcal{A}_{2}$, $\varphi
_{1}$ a minimal admissible $\mathcal{B}_{1}$-projection on $\mathcal{A}
_{1}$ and $\varphi _{2}$ a minimal admissible $\mathcal{B}_{2}$-projection on $\mathcal{A}_{2}$. If $\phi:\mathcal{B}_{1}\rightarrow 
\mathcal{B}_{2}$ is a unital local isometric $\ast $-isomorphism, then it
extends to a unique unital local isometric $\ast $-isomorphism $\widetilde{\phi }:C^{\ast }\left( \varphi _{1}\right) \rightarrow C^{\ast }\left(
\varphi _{2}\right).$
\end{lemma}

\begin{proof}
Since $\phi :\mathcal{B}_{1}\rightarrow \mathcal{B}_{2}$ is a unital local
isometric $\ast $-isomorphism, $\phi $ is invertible and $\phi ^{-1}:\mathcal{B}
_{2}\rightarrow \mathcal{B}_{1}$ is a unital local isometric $\ast $
-isomorphism too. By Lemma \ref{projection}, $C^{\ast }(\varphi _{1})$ and $
C^{\ast }(\varphi _{2})$ are admissible injective, and since $\phi $ and $
\phi ^{-1}$ are unital admissible local $\mathcal{CP}$-maps, there exist the unital admissible
 local $\mathcal{CP}$-maps $\tilde{\phi}:C^{\ast }(\varphi
_{1})\rightarrow C^{\ast }(\varphi _{2})$ and $\widetilde{\phi ^{-1}}
:C^{\ast }(\varphi _{2})\rightarrow C^{\ast }(\varphi _{1})$ such that 
\begin{equation*}
 \tilde{\phi}\restriction _{\mathcal{B}_{1}}=\phi \text{ and }
\widetilde{\phi ^{-1}}\restriction _{\mathcal{B}_{2}}=\phi ^{-1}.
\end{equation*}

On the other hand, $\widetilde{\phi ^{-1}}\circ \tilde{\phi}:C^{\ast
}(\varphi _{1})\rightarrow C^{\ast }(\varphi _{1})$ and $\tilde{\phi}\circ 
\widetilde{\phi ^{-1}}:C^{\ast }(\varphi _{2})\rightarrow C^{\ast }(\varphi
_{2})$ are unital admissible local $\mathcal{CP}$-maps such that 
\begin{equation*}
 \widetilde{\phi ^{-1}}\circ \tilde{\phi}\restriction _{\mathcal{B}
_{1}}=id_{\mathcal{B}_{1}}\text{ and } \tilde{\phi}\circ 
\widetilde{\phi ^{-1}}\restriction _{\mathcal{B}_{2}}=id_{\mathcal{B}_{2}}.
\end{equation*}
Then, by Lemma \ref{p1}, $\widetilde{\phi ^{-1}}\circ \tilde{\phi}=id_{C^{\ast }(\varphi _{1})}$ and $\tilde{\phi}\circ \widetilde{\phi ^{-1}}=id_{C^{\ast }(\varphi _{2})}$. Consequently, there exists $\left( \tilde{\phi
}\right) ^{-1}=\widetilde{\phi ^{-1}}$, and since $\tilde{\phi}$ and $\widetilde{\phi^{-1}}$ are unital admissible local $\mathcal{CP}$-maps, by Lemma 
\ref{U}, it follows that $\tilde{\phi}$ is a unital local isometric $\ast $-isomorphism.

To prove the uniqueness, let $\Psi :C^{\ast }\left( \varphi _{1}\right)
\rightarrow C^{\ast }\left( \varphi _{2}\right) $ be another unital local
isometric $\ast $-isomorphism such that $ \Psi \restriction _{\mathcal{B}_{1}}=\phi $. Then $\widetilde{\phi ^{-1}}\circ \Psi $ is a unital admissible local $\mathcal{CP}$-map from $C^{\ast }\left( \varphi
_{1}\right) $ to $C^{\ast }\left( \varphi _{1}\right) $ such that $ 
\widetilde{\phi ^{-1}}\circ \Psi \restriction _{\mathcal{B}_{1}}=\phi
^{-1}\circ \phi =id_{\mathcal{B}_{1}}$. By Lemma \ref{p1}, it
follows that $\widetilde{\phi ^{-1}}\circ \Psi =$id$_{C^{\ast }(\varphi
_{1})}$, and so $\Psi =$ $\tilde{\phi}$.
\end{proof}

\subsection*{The proof of Theorem \ref{p}}

\begin{proof}
Since $\mathcal{A}$ is a unital Fr\'{e}chet locally $C^{\ast }$-algebra,
there exists a unital admissible injective Fr\'{e}chet locally $C^{\ast }$
-algebra $\mathcal{C}$ such that $\mathcal{A}\subseteq \mathcal{C}$. By
Theorem \ref{Help}, there exists a minimal admissible $\mathcal{A}$-projection $\varphi $ on $\mathcal{C}$. Let $\mathcal{B}:=C^{\ast }(\varphi
)$ and let $\phi $ be the inclusion of $ \mathcal{A}$ into $\mathcal{B}$.
Clearly, $(\mathcal{B},\phi )$ is an extension of$\ \mathcal{A}$.
By Lemma \ref{projection}, $\mathcal{B}$ is a unital admissible injective locally $C^{\ast }$-algebra, and by Lemma \ref{p1}, $(\mathcal{B},\phi )$ is
an admissible injective envelope of $\mathcal{A}$.

Let $(\mathcal{B}_{1},\phi _{1})$ be another admissible injective envelope
of $\mathcal{A}$. Then id$_{\mathcal{B}_{1}}$ is the unique unital admissible local $\mathcal{CP}$-map from $\mathcal{B}_{1}$ to $\mathcal{B}_{1}$
which fixes each element in $\phi _{1}(\mathcal{A})$. Therefore, id$_{
\mathcal{B}_{1}}$ is a minimal admissible $\phi _{1}(\mathcal{A})$
-projection. Then, since $\phi _{1}:\mathcal{A}\rightarrow \phi _{1}(
\mathcal{A})$ is a unital local isometric $\ast $-isomorphism, by Lemma \ref{p2}, it extends to a unital local isometric $\ast $-isomorphism $\Phi:
\mathcal{B}\rightarrow \mathcal{B}_{1}$. Moreover, $ \left( \Phi \circ
\phi \right) (a)=\phi _{1}(a)$ for all $a\in \mathcal{A},$ and so $\Phi
\circ \phi =\phi _{1}$.
\end{proof}

\begin{remark}
\begin{itemize}
\item[(1)] The admissible injective envelope of a $C^*$-algebra $\mathcal{A}$ coincides with its injective envelope \cite{H}.
\item[(2)] The admissible injective envelope of a unital admissible injective Fr\'{e}chet locally $C^*$-algebra $\mathcal{A}$ is $\mathcal{A}$.
\end{itemize}

\end{remark}

\begin{corollary}
Let $\mathcal{A}$ be a unital Fr\'{e}chet locally $C^{\ast }$-algebra and let $(
\mathcal{B},\phi )$ be its admissible injective envelope. Then, for each
unital local isometric $\ast $-automorphism $\Phi $ of $\mathcal{A}$, there
exists a unique unital local isometric $\ast $-automorphism $\widehat{\Phi }$ of $
\mathcal{B}$ such that $\phi \circ \Phi =\widehat{\Phi }\circ \phi $.
\end{corollary}

\begin{proof}
The pair $(\mathcal{B},\phi \circ \Phi )\ $ is an injective admissible extension of $\mathcal{A}$. Moreover, if $\psi:\mathcal{B\rightarrow B}$
is a unital admissible local $\mathcal{CP}$-map from $\mathcal{B}$ to $
\mathcal{B}$ which fixes each element in $\left( \phi \circ \Phi \right) (
\mathcal{A})$, then since $\Phi $ is a unital local isometric $\ast $
-isomorphism, $\psi $ fixes each element in $\phi (\mathcal{A})$, and so $
\psi =id_{\mathcal{B}}$. Therefore, $(\mathcal{B},\phi \circ \Phi )\ $ is
an admissible injective envelope for $\mathcal{A}$, and then there exists a
unique unital local isometric $\ast $-automorphism $\widehat{\Phi }$ of $\mathcal{B}$ such that $\widehat{\Phi }\circ \phi =\phi \circ \Phi $.
\end{proof}

\begin{example}
Let $\lbrace\mathcal{H}; \mathcal{E}=\lbrace \mathcal{H}_{n} \rbrace_{n \geq 1}; 
\mathcal{D}_{\mathcal{E}}\rbrace$ be a Fr\'{e}chet quantized domain in the
Hilbert space $\mathcal{H}$. We denote by $B(\mathcal{H})$ the $C^{\ast }$
-algebra of all bounded linear operators on $\mathcal{H}$, and by $K(
\mathcal{H})$ the $C^{\ast }$-algebra of all compact operators on $\mathcal{H}$. For each $\xi ,\eta \in \mathcal{D}_{\mathcal{E}}$, the rank one
operator $\theta _{\xi ,\eta }:\mathcal{D}_{\mathcal{E}}\rightarrow \mathcal{%
D}_{\mathcal{E}}, \theta _{\xi ,\eta }\left( \zeta \right) =\xi \left\langle
\eta ,\zeta \right\rangle $ is an element in $C^{\ast }(\mathcal{D}_{\mathcal{E}})$. The closure of the linear space generated by the rank one
operators is a locally $C^{\ast}$-subalgebra of $C^{\ast }(\mathcal{D}_{\mathcal{E}})$, denoted by $K^{\ast }(\mathcal{D}_{\mathcal{E}})$ (for more
details see \cite{KP} and \cite{I}). It is known that $B(\mathcal{H})$ is an
injective envelope for $K(\mathcal{H})$. We have a similar result in the
locally convex case.

For each $n\geq 1$, $\mathcal{P}_{n}$ denotes the orthogonal projection of $
\mathcal{H}$ onto $\mathcal{H}_{n}$. Then $T\in C^{\ast }(\mathcal{D}_{\mathcal{E}})$ if and only if 
\begin{equation*}
T=\sum\limits_{n=1}^{\infty } \left( \text{id}_{\mathcal{H}}-\mathcal{P}_{n-1}\right) \mathcal{P}_{n}T\left( \text{id}_{\mathcal{H}}-\mathcal{P}
_{n-1}\right) \mathcal{P}_{n}\restriction _{\mathcal{D}_{\mathcal{E}}},
\end{equation*}
where $\mathcal{P}_{0}=0$ \cite[Proposition 4.2]{D}. Therefore, for each $
n\geq 1$, 
\begin{equation*}
\left( C^{\ast }(\mathcal{D}_{\mathcal{E}})\right) _{n}=B(\mathcal{H}
_{1})\oplus B(\left( \mathcal{H}_{1}\right) ^{\bot }\cap \mathcal{H}
_{2})\oplus \cdot \cdot \cdot \oplus B(\left( \mathcal{H}_{n-1}\right)
^{\bot }\cap \mathcal{H}_{n})
\end{equation*}
and 
\begin{equation*}
\left( K^{\ast }(\mathcal{D}_{\mathcal{E}})\right) _{n}=K(\mathcal{H}
_{1})\oplus K(\left( \mathcal{H}_{1}\right) ^{\bot }\cap \mathcal{H}
_{2})\oplus \cdot \cdot \cdot \oplus K(\left( \mathcal{H}_{n-1}\right)
^{\bot }\cap \mathcal{H}_{n}).
\end{equation*}
It is known that the injective envelope of the direct sum $\mathcal{A}
_{1}\oplus \mathcal{A}_{2}$ of the $C^{\ast }$-algebras $\mathcal{A}_{1}$
and $\mathcal{A}_{2}$ is the direct sum $\mathcal{I}(\mathcal{A}
_{1})\oplus \mathcal{I}(\mathcal{A}_{2})$ of the injective envelopes $
\mathcal{I}(\mathcal{A}_{1})$ and $\mathcal{I}(\mathcal{A}_{2})$ of $
\mathcal{A}_{1}$ and $\mathcal{A}_{2}$, respectively. Therefore, 

\begin{align*}
\mathcal{I}\left( \left( K^{\ast }(\mathcal{D}_{\mathcal{E}})\right)_{n}\right)  
&= \mathcal{I}\left( K(\mathcal{H}_{1})\right) \oplus \mathcal{I}\left( K(\left( \mathcal{H}_{1}\right)^{\bot }\cap \mathcal{H}_{2})\right) \oplus \cdots \oplus \mathcal{I}\left( K(\left( \mathcal{H}_{n-1}\right)^{\bot }\cap \mathcal{H}_{n})\right)  \\
&= B(\mathcal{H}_{1})\oplus B(\left( \mathcal{H}_{1}\right)^{\bot }\cap \mathcal{H}_{2})\oplus \cdots \oplus B(\left( \mathcal{H}_{n-1}\right)^{\bot }\cap \mathcal{H}_{n}) \\
&= \left( C^{\ast }(\mathcal{D}_{\mathcal{E}})\right)_{n}.
\end{align*}


Let $\varphi:C^{\ast }(\mathcal{D}_{\mathcal{E}})\rightarrow C^{\ast }(
\mathcal{D}_{\mathcal{E}})$ be a unital admissible local $\mathcal{CP}$-map
such that $ \varphi \restriction _{K^{\ast }(\mathcal{D}_{\mathcal{E}})}= id_{K^{\ast }(\mathcal{D}_{\mathcal{E}})}$. Then, for each $n\geq 1$, there exists a unital $\mathcal{CP}$-map $\varphi _{n}:\left( C^{\ast }(\mathcal{D}_{\mathcal{E}})\right) _{n}\rightarrow \left( C^{\ast }(\mathcal{D}_{\mathcal{E}})\right) _{n}$ such that $\varphi _{n}\left( 
T\restriction _{\mathcal{H}_{n}}\right) = \varphi \left( T\right)
\restriction _{\mathcal{H}_{n}}\ \textit{and}\  \varphi _{n}\restriction _{\left(
K^{\ast }(\mathcal{D}_{\mathcal{E}})\right) _{n}}=id_{\left( K^{\ast}(\mathcal{D}_{\mathcal{E}})\right) _{n}}.$ Since $\left( C^{\ast }(\mathcal{D}
_{\mathcal{E}})\right) _{n}$ is an injective envelope of $\left( K^{\ast }(
\mathcal{D}_{\mathcal{E}})\right) _{n},\varphi _{n}=id_{\left( C^{\ast }(\mathcal{D}_{\mathcal{E}})\right) _{n}}$. Therefore, $\varphi =id_{C^{\ast
}(\mathcal{D}_{\mathcal{E}})}$, and so $C^{\ast }(\mathcal{D}_{\mathcal{E}})$
is an admissible injective envelope for $K^{\ast }(\mathcal{D}_{\mathcal{E}
}).$
\end{example}

\begin{example}
It is known that the injective envelope of a $C^{\ast }$-algebra $\mathcal{A}
$ that contains $K(\mathcal{H})$, the $C^{\ast }$-algebra of all compact
operators on a Hilbert space $\mathcal{H}$, is $B(\mathcal{H})$. We have a
similar result for unital Fr\'{e}chet locally $C^{\ast }$-algebras. Let $\lbrace
\mathcal{H}; \mathcal{E}=\lbrace \mathcal{H}_{n} \rbrace_{n \geq 1}; \mathcal{D}_{\mathcal{E}}\rbrace$ be a Fr\'{e}chet quantized domain in the Hilbert space $\mathcal{H}$. If $\mathcal{A}$ is a unital Fr\'{e}chet locally $C^{\ast }$-algebra that contains $K^{\ast }(\mathcal{D}_{\mathcal{E}})$, then $C^{\ast
}(\mathcal{D}_{\mathcal{E}})$ is its admissible injective envelope.
Moreover, $\lbrace\mathcal{I}\left( \mathcal{A}_{n}\right) \rbrace _{n\geq 1}$, where $\mathcal{I}\left( \mathcal{A}_{n}\right) $ is the injective envelope
of $\mathcal{A}_{n}$, is an inverse system of  $C^{\ast}$-algebras and
its inverse limit is an admissible injective envelope of $\mathcal{A}$.
\end{example}

Let $\mathcal{B}$ be a unital locally $C^{\ast }$-algebra whose topology is
defined by the family of $C^{\ast}$-seminorms $\lbrace q_{\lambda }\rbrace _{\lambda
\in \Lambda }$ and let $\varphi :\mathcal{A}\rightarrow \mathcal{B}$ be a unital
linear map such that $q_{\lambda }\left( \varphi \left( a\right) \right)
=p_{\lambda }\left( a\right) $, for all $a\in \mathcal{A}$ and for all $\lambda \in \Lambda $. Then $\varphi \left( \mathcal{A}\right) $, the range
of $\varphi $, is a closed subspace of $\mathcal{B}$ and there exists a
unital linear map $\varphi ^{-1}:$ $\varphi \left( \mathcal{A}\right)
\rightarrow \mathcal{A}$ such that $\varphi ^{-1}\circ \varphi =id_{\mathcal{A}}$. Moreover, for each $\lambda \in \Lambda $,
\begin{equation*}
p_{\lambda }\left( \varphi ^{-1}\left( b\right) \right) =p_{\lambda }\left(
\varphi ^{-1}\left( \varphi \left( a\right) \right) \right) =p_{\lambda
}\left( a\right) =q_{\lambda }\left( \varphi \left( a\right) \right)
=q_{\lambda }\left( b\right),
\end{equation*}
for all $b\in $ $\varphi \left( \mathcal{A}\right) $. If $\varphi $ and 
$\varphi ^{-1}$ are unital local $\mathcal{CP}$-maps we say that $\varphi $ is \textit{local completely isometric}.

Note that if $(\mathcal{B},\phi )$ is an admissible extension of $\mathcal{A}$, then $\phi $ is a local completely isometric map.

\begin{definition}
An admissible extension $(\mathcal{B},\phi )$ of a unital locally $C^{\ast }$-algebra $\mathcal{A}$ is essential if for any unital admissible local
completely positive map $\varphi $ from $\mathcal{B}$ to another unital
locally $C^{\ast }$-algebra $\mathcal{C}$, $\varphi $ is local completely
isometric whenever $\varphi \circ \phi $ is local completely isometric.
\end{definition}

As in the case of $C^{\ast }$-algebras we obtain a characterization of an admissible injective envelope for a unital Fr\'{e}chet locally $C^{\ast }$-algebra in terms of its admissible extensions.

\begin{theorem}\label{pp}
 An admissible extension $(\mathcal{B},\phi )$ of a unital Fr\'{e}chet locally $C^{\ast }$-algebra $\mathcal{A}$ is an admissible injective
envelope if and only if it is admissible injective and essential.
\end{theorem}

\begin{proof}
The proof of this theorem is similar to the proof of \cite[Proposition 4.7.]{H}. First, we assume that $(\mathcal{B},\phi )$ is an admissible
injective envelope for $\mathcal{A}$. Then $(\mathcal{B},\phi )\ $ is an admissible injective extension. Let $\mathcal{C}$ be a locally $C^{\ast }$-algebra and $\varphi:\mathcal{B}\rightarrow \mathcal{C} $ be a unital
admissible local $\mathcal{CP}$-map such that $\varphi \circ \phi $ is local
completely isometric. Then $\phi \circ \left( \varphi \circ \phi \right)
^{-1}:\left( \varphi \circ \phi \right) \left( \mathcal{A}\right)
\rightarrow \mathcal{B}$ is a unital admissible local $\mathcal{CP}$-map,
and since $\mathcal{B}$ is admissible injective, there exists a unital
admissible local $\mathcal{CP}$-map $\psi :\mathcal{C}\rightarrow \mathcal{B}
$ such that
\begin{equation*}
 \psi \restriction _{\left( \varphi \circ \phi \right) \left( \mathcal{A}\right) }=\phi \circ \left( \varphi \circ \phi \right) ^{-1}.
\end{equation*}
From the last relation, we deduce that 
\begin{equation*}
 \psi \circ \varphi \restriction _{\phi \left( \mathcal{A}\right) }=\text{id}_{\phi \left( \mathcal{A}\right) },
\end{equation*}
and since $(\mathcal{B},\phi )$ is an admissible injective envelope for $\mathcal{A}$, it follows that $\psi \circ \varphi $ $=id_{\mathcal{B}}$.
Therefore, $\varphi $ is local completely isometric.

Conversely, suppose that $(\mathcal{B},\phi )$ is an essential admissible
injective extension of $\mathcal{A}$. Let $(\mathcal{C},\chi )$ be an admissible injective envelope for $\mathcal{A}$. Then there exists a unital admissible local $\mathcal{CP}$-map $\varphi:\mathcal{B}\rightarrow 
\mathcal{C}$ such that $\varphi \circ \phi =\chi $. Consequently, $\varphi
\circ \phi $ is local completely isometric, and since the admissible extension $(\mathcal{B},\phi )$ is essential, $\varphi $ is local completely
isometric. On the other hand, since $\mathcal{B}$ is admissible injective, $
\varphi ^{-1}:\varphi \left( \mathcal{B}\right) \rightarrow \mathcal{B}$
extends to a unital admissible local $\mathcal{CP}$-map $\widetilde{\varphi
^{-1}}:\mathcal{C}\rightarrow \mathcal{B}$. Moreover, $ \varphi \circ 
\widetilde{\varphi ^{-1}}\restriction _{\chi \left( \mathcal{A}\right) }=id
_{\chi \left( \mathcal{A}\right) }$, since
\begin{equation*}
\left( \varphi \circ \widetilde{\varphi ^{-1}}\right) \left( \chi \left(
a\right) \right) =\left( \varphi \circ \widetilde{\varphi ^{-1}}\circ
\varphi \circ \phi \right) \left( a\right) =\left( \varphi \circ \phi
\right) \left( a\right) =\chi \left( a\right),
\end{equation*}
for all $a\in \mathcal{A}$. From this relation and taking into account that 
$(\mathcal{C},\chi )$ is an admissible injective envelope for $\mathcal{A}$,
we conclude that $\varphi \circ \widetilde{\varphi ^{-1}}=id_{\mathcal{C}}$. Therefore, $\varphi $ is bijective, and since $\varphi $ and $\varphi
^{-1}$ are unital admissible local $\mathcal{CP}$-maps, by Lemma \ref{U}, $\varphi $ is a local isometric $\ast $-isomorphism. Consequently, $(\mathcal{B},\phi )$ is an admissible injective envelope for $\mathcal{A}$.
\end{proof}

\begin{corollary}
A unital Fr\'{e}chet locally $C^{\ast}$-algebra $\mathcal{A}$ is admissible
injective if and only if it has no proper essential admissible extension.
\end{corollary}

\begin{proof}
First, we assume that $\mathcal{A}$ is admissible injective. Let $\left( 
\mathcal{B},\phi \right) $ be an essential admissible extension for $\mathcal{A}$. Since $\phi $ is a local completely isometric linear map and
 $\mathcal{A}$ is admissible injective, there exists a unital
admissible local $\mathcal{CP}$-map $\widetilde{\phi ^{-1}}:\mathcal{B}
\rightarrow \mathcal{A}$ such that $\widetilde{\phi ^{-1}}\circ \phi =id_{\mathcal{A}}$. Consequently, $\widetilde{\phi ^{-1}}\circ \phi $ is a local
completely isometric linear map, and since $\left( \mathcal{B},\phi \right) $
is an essential admissible extension for $\mathcal{A},$ $\widetilde{\phi
^{-1}}$ is a local completely isometric linear map too, and so, there exists a unital admissible local $\mathcal{CP}$-map $\left( \widetilde{\phi ^{-1}}\right) ^{-1}\  $ such that $\left( 
\widetilde{\phi ^{-1}}\right) ^{-1}\circ \widetilde{\phi ^{-1}}=id_{
\mathcal{B}}$. Then, $\left( \widetilde{\phi ^{-1}}\right) ^{-1}=\phi .$
Hence, $\phi $ is bijective, and by Lemma \ref{U}, $\phi $ is a local
isometric $\ast $-isomorphism.

To prove the converse implication, suppose that $\mathcal{A}$ is not admissible injective and let $\left( \mathcal{B},\phi \right) $ be an admissible injective envelope for $\mathcal{A}$. Then, by Theorem \ref{pp}, $\left( \mathcal{B},\phi \right) $ is an essential admissible  extension for $\mathcal{A}$, which is a contradiction.
\end{proof}

\section{Admissible injective envelopes for Fr\'{e}chet locally $C^{\ast }$-algebras via inverse limits of injective envelopes for $C^{\ast }$-algebras}

In this section we investigate how an admissible injective envelope for a unital Fr\'{e}chet locally $C^{\ast}$-algebra can be realized via an
inverse limit of injective envelopes for $C^{\ast}$-algebras.\\

Let $\mathcal{A}$ be a unital Fr\'{e}chet locally $C^{\ast }$-algebra and $\lbrace
\mathcal{A}_{n}; \pi _{mn}^{\mathcal{A}}:\mathcal{A}_{m}\rightarrow \mathcal{A}_{n}, m\geq n\rbrace$ be its Arens-Michael decomposition. By Remark \ref{i}, $\mathcal{A}$ can be regard as a locally $C^{\ast }$-subalgebra of a unital admissible injective Fr\'{e}chet locally $C^{\ast }$-algebra $\mathcal{B}$, and then, by Theorem \ref{p}, there exists a minimal
admissible $\mathcal{A}$-projection $\varphi $ on $\mathcal{B}$ and $C^{\ast
}\left( \varphi \right) $ is an admissible injective envelope for $\mathcal{A}$.

Since $\varphi :\mathcal{B}\rightarrow \mathcal{B}$ is a unital admissible local completely positive map, by Remark \ref{a}, $\varphi =\varprojlim\limits_{n}\varphi _{n},$ where $\varphi _{n}$ is a unital
completely positive map on $\mathcal{B}_{n}$ such that 
\begin{equation*}
\varphi _{n}\left( \pi _{n}^{\mathcal{B}}\left( b\right) \right) =\pi _{n}^{
\mathcal{B}}\left( \varphi \left( b\right) \right), 
\end{equation*}
for all $b\in \mathcal{B}, n\geq 1$. It is easy to check that for each $n\geq
1, \varphi _{n}$ is an $\mathcal{A}_{n}$-projection on $\mathcal{B}_{n}.$

According to Remark \ref{Help1}, since $\varphi $ is a minimal admissible $\mathcal{A}$-projection  on $\mathcal{B}, \lbrace p_{n}\circ \varphi
\rbrace _{n\geq 1}$ is a minimal family of $\mathcal{A}$ -seminorms on $\mathcal{B}
$, and then, by the proof of Proposition \ref{4}, for each $n\geq 1,$ 
\begin{equation*}
\left( p_{n}\circ \varphi \right) \left( \cdot \right) =p_{n}^{\min }\left(
\cdot \right) =p_{n,\mathcal{B}_{n}}^{\min }\left( \pi _{n}^{\mathcal{B}
}\left( \cdot \right) \right), 
\end{equation*}
where $p_{n,\mathcal{B}_{n}}^{\min }$ is the minimal $\mathcal{A}_{n}$-seminorm on $\mathcal{B}_{n}$. On the other hand, for each $n\geq 1,$
\begin{equation*}
p_{n,\mathcal{B}_{n}}^{\min }\left( \pi _{n}^{\mathcal{B}}\left( \cdot
\right) \right) =\left( p_{n}\circ \varphi \right) \left( \cdot \right)
=\left\Vert \varphi _{n}\left( \pi _{n}^{\mathcal{B}}\left( \cdot \right)
\right) \right\Vert _{\mathcal{B}_{n}}.
\end{equation*}
It follows that $\varphi _{n}$ is a minimal $\mathcal{A}_{n}$-projection on $\mathcal{B}_{n}$ \cite[Remark 3.9]{H}. Therefore, for
each $n\geq 1$, $C^{\ast }\left( \varphi _{n}\right) $ is an injective
envelope of $\mathcal{A}_{n}$ \cite[Theorem 3.4]{H}. Since

\begin{eqnarray*}
\pi _{mn}^{\mathcal{B}}\left( \pi _{m}^{\mathcal{B}}\left( b_{1}\right)
\cdot \pi _{m}^{\mathcal{B}}\left( b_{2}\right) \right)  &=&\pi _{mn}^{
\mathcal{B}}\left( \varphi _{m}\left( \pi _{m}^{\mathcal{B}}\left(
b_{1}\right) \pi _{m}^{\mathcal{B}}\left( b_{2}\right) \right) \right)  \\
&=&\pi _{mn}^{\mathcal{B}}\left( \varphi _{m}\left( \pi _{m}^{\mathcal{B}
}\left( b_{1}b_{2}\right) \right) \right) =\pi _{mn}^{\mathcal{B}}\left( \pi
_{m}^{\mathcal{B}}\left( \varphi \left( b_{1}b_{2}\right) \right) \right)  \\
&=&\pi _{n}^{\mathcal{B}}\left( \varphi \left( b_{1}b_{2}\right) \right)
=\varphi _{n}\left( \pi _{n}^{\mathcal{B}}\left( b_{1}b_{2}\right) \right) 
\\
&=&\varphi _{n}\left( \pi _{n}^{\mathcal{B}}\left( b_{1}\right) \pi _{n}^{\mathcal{B}}\left( b_{2}\right) \right) =\pi _{n}^{\mathcal{B}}\left(
b_{1}\right) \cdot \pi _{n}^{\mathcal{B}}\left( b_{2}\right)  \\
&=&\pi _{mn}^{\mathcal{B}}\left( \pi _{m}^{\mathcal{B}}\left( b_{1}\right)
\right) \cdot \pi _{mn}^{\mathcal{B}}\left( \pi _{m}^{\mathcal{B}}\left(
b_{2}\right) \right), 
\end{eqnarray*}
for all $b_{1},b_{2}\in Im\left( \varphi \right) $, we deduce that $\left\lbrace C^{\ast }\left( \varphi _{n}\right); \pi _{mn}^{\mathcal{B}
}\restriction _{C^{\ast}\left( \varphi _{m}\right) }; m\geq n \right\rbrace$ is an
inverse system of $C^{\ast}$-algebras. Moreover, the map $a\mapsto \left(
\varphi _{n}\left( \pi_{n}^{\mathcal{B}}(a)\right) \right) _{n\geq 1}$ from $C^{\ast }\left(
\varphi \right) $ to $\varprojlim\limits_{n}C^{\ast }\left(
\varphi _{n}\right) $ is a unital local isometric $\ast $-morphism.
Therefore, the admissible injective envelope of a unital Fr\'{e}chet locally 
$C^{\ast }$-algebra can be identified with the inverse limit of the
injective envelopes for its Arens-Michael decomposition.

If we denote by $\mathcal{I}(\mathcal{A})$ the admissible injective envelope for $\mathcal{A}$, then we have the following result:

\begin{proposition}
Let $\mathcal{A}$ be a unital Fr\'{e}chet locally 
$C^{\ast }$-algebra. Then 
\begin{itemize}
\item[(1)] For each $n\geq 1$, the $C^*$-algebras $\left( \mathcal{I}(\mathcal{A}) \right)_{n}$ and $\left( \mathcal{I}(\mathcal{A}_{n}) \right)$ are isomorphic.
\item[(2)] $\mathcal{I}(\mathcal{A})=\varprojlim\limits_{n}\mathcal{I}(\mathcal{A}_{n}).$
\item[(3)] If $\mathcal{A}$ is a Fr\'{e}chet locally $W^*$-algebra, then $\mathcal{A}$ is injective if and only if the $C^*$-algebras $\mathcal{A}_{n}, n\geq 1$, are injective.
\end{itemize}

\end{proposition}

\begin{proof}
According to the above comments, there exist a unital admissible injective Fr\'{e}chet locally $C^*$-algebra $\mathcal{B}$ such that $\mathcal{A}$ can be identified with a locally $C^*$-subalgebra of $\mathcal{B}$, and a minimal admissible $\mathcal{A}$-projection $\varphi$ on $\mathcal{B}$ such that $\mathcal{I}(\mathcal{A})$ can be identified with the range of $\varphi$, $C^*(\varphi)$. We see that $\varphi=\varprojlim\limits_{n}\varphi_{n}$, where $\varphi_{n}$ is a minimal $\mathcal{A}_{n}$-projection on $\mathcal{B}_{n}$ and $C^*(\varphi)=\varprojlim\limits_{n} C^*(\varphi_{n})$.

$(1)$ Let $n\geq 1$. To prove that the $C^*$-algebras $\left( \mathcal{I}(\mathcal{A}) \right)_{n}$ and $\mathcal{I}(\mathcal{A}_{n})$ are isomorphic it is sufficient to show that the $C^*$-algebras $C^*(\varphi)/\ker p_{n}\restriction_{C^*(\varphi)}$ and $C^*(\varphi_{n})$ are isomorphic. We consider the map $\phi_{n}:C^*(\varphi)/\ker p_{n}\restriction_{C^*(\varphi)}\rightarrow C^*(\varphi_{n})$ given by $$\phi_{n}\left( \varphi(b)+\ker p_{n}\restriction_{C^*(\varphi)} \right)=\varphi_{n}\left( \pi_{n}^{\mathcal{B}}(b) \right).$$
Since $\varphi_{n}\circ\pi_{n}^{\mathcal{B}}=\pi_{n}^{\mathcal{B}}\circ\varphi$, the map $\phi_{n}$ is correct defined. Clearly, $\phi_{n}$ is a bijective linear map. From $$\phi_{n}\left( \varphi(b)^*+\ker p_{n}\restriction_{C^*(\varphi)} \right)=\phi_{n}\left( \varphi(b^*)+ \ker p_{n}\restriction_{C^*(\varphi)} \right)=\varphi_{n}\left( \pi_{n}^{\mathcal{B}}(b^*) \right)$$
$$=\varphi_{n}\left( \pi_{n}^{\mathcal{B}}(b) \right)^*=\left(  \phi_{n}\left( \varphi(b)^*+\ker p_{n}\restriction_{C^*(\varphi)} \right) \right)^*$$
and
$$\phi_{n}\left( \varphi(b_{1})\cdot\varphi(b_{2})+ \ker p_{n}\restriction_{C^*(\varphi)} \right)=\phi_{n}\left( \varphi\left(  \varphi(b_{1})\varphi(b_{2})  \right)+ \ker p_{n}\restriction_{C^*(\varphi)} \right)$$
$$=\varphi_{n}\left( \pi_{n}^{\mathcal{B}}\left( \varphi(b_{1})\varphi(b_{2}) \right)  \right)=\varphi_{n}\left( \pi_{n}^{\mathcal{B}}\left(\varphi(b_{1}) \right)\pi_{n}^{\mathcal{B}}\left(\varphi(b_{2}) \right)  \right)=\varphi_{n}\left(\varphi_{n}\left( \pi_{n}^{\mathcal{B}}(b_{1}) \right)\varphi_{n}\left( \pi_{n}^{\mathcal{B}}(b_{2}) \right)     \right)$$
$$=\varphi_{n}\left( \pi_{n}^{\mathcal{B}}(b_{1})  \right)\cdot\varphi_{n}\left( \pi_{n}^{\mathcal{B}}(b_{2})  \right)=\phi_{n}\left( \varphi(b_{1})+ \ker p_{n}\restriction_{C^*(\varphi)} \right)\cdot\phi_{n}\left( \varphi(b_{2})+ \ker p_{n}\restriction_{C^*(\varphi)} \right)$$
for all $b, b_{1}, b_{2}\in\mathcal{B} $, we deduce that $\phi_{n}$ is a $*$-morphism. Therefore, $\phi_{n}$ is a $C^*$-isomorphism.

$(2)$ It follows from $(1)$.

$(3)$ Suppose that $\mathcal{A}$ is a Fr\'{e}chet locally $W^*$-algebra. Then $\mathcal{A}$ is injective if and only if it is admissible injective (see Remark \ref{Rem2}). On the other hand, by $(1)$ and $(2)$, $\mathcal{A}$ is admissible injective if and only if for each $n\geq 1$, the $C^*$-algebras $\mathcal{A}_{n}$ and $\mathcal{I}(\mathcal{A}_{n})$ are isomorphic. Consequently, $\mathcal{A}$ is injective if and only if the $C^*$-algebras $\mathcal{A}_{n}$, $n\geq 1$, are injective.

\end{proof}

\subsection*{Funding} The first author was supported by a grant from UEFISCDI, a project
number PN-III-P4-PCE-2021-0282

\subsection*{Data Availability} Data sharing not applicable to this article as no datasets were generated or analysed during the current study.
\subsection*{Declarations}
\subsection*{Conflicts of interests} The authors have no relevant financial or non-financial interests to disclose.

\end{document}